%% file: Set-theoretic_geology.tex
\documentclass{amsart}
\usepackage{amssymb,bbm}
\input MathMacrosJDH
\usepackage{diagrams}
\diagramstyle[tight,centredisplay,dpi=600,noPostScript,heads=LaTeX]%
\usepackage[hidelinks]{hyperref}
\usepackage{graphics}
\usepackage{graphicx}

\newcommand{\ro}[1]{{\rm ro}{\left(#1\right)}}
\newcommand{\CCA}{{\rm CCA}}
\newcommand{\V}{V}
\newcommand{\sub}{\of}

\newcommand{\reals}{\mathord{\mathbb{R}}}

\begin{document}
\author{Gunter Fuchs}
\address{G. Fuchs, Department of Mathematics, The College of Staten Island of CUNY, 2800 Victory Boulevard,
Staten Island, NY 10314} \email{Gunter.Fuchs@csi.cuny.edu}
\author{Joel David Hamkins}
\address{J. D. Hamkins, Mathematics, The Graduate Center of The City University of New York, 365 Fifth Avenue, New York, NY 10016
 \& Mathematics, The College of Staten Island of CUNY}
\email{jhamkins@gc.cuny.edu, http://jdh.hamkins.org}
\author{Jonas Reitz}
\address{J. Reitz, New York City College of Technology of The City University of New York, Mathematics, 300 Jay Street, Brooklyn, NY 11201}
\email{jonasreitz@gmail.com}
\today  
\thanks{The research of the second author has been supported in part by NSF grant DMS-0800762, PSC-CUNY grant 64732-00-42, CUNY Collaborative Incentive Award 80209-06 20, Simons Foundation grant 209252, and
Netherlands Organization for Scientific Research (NWO) grant B 62-619, and
he is grateful to the Institute of Logic, Language and
Computation at Universiteit van Amsterdam for the support
of a Visiting Professorship during his sabbatical there in
2007. The main part of this research was largely completed
in 2007-8, but the preparation and subsequent publication of the paper was
unfortunately delayed. The paper was submitted to this journal August 5, 2011, and the referee report was received September 25, 2014.}

 \hyphenation{multi-verse}

\begin{abstract}
A {\df ground} of the universe $V$ is a transitive proper
class $W\of V$, such that $W\satisfies\ZFC$ and $V$ is
obtained by set forcing over $W$, so that $V=W[G]$ for some
$W$-generic filter $G\of\P\in W$. The model $V$ satisfies
the ground axiom \GA\ if there are no such $W$ properly
contained in $V$. The model $W$ is a {\df bedrock} of $V$
if $W$ is a ground of $V$ and satisfies the ground axiom.
The {\df mantle} of $V$ is the intersection of all grounds
of $V$. The {\df generic mantle} of $V$ is the intersection
of all grounds of all set-forcing extensions of $V$. The
generic \HOD, written \gHOD, is the intersection of all
{\HOD}s of all set-forcing extensions. The generic \HOD\ is
always a model of \ZFC, and the generic mantle is always a
model of \ZF. Every model of \ZFC\ is the mantle and
generic mantle of another model of \ZFC. We prove this
theorem while also controlling the \HOD\ of the final
model, as well as the generic \HOD. Iteratively taking the
mantle penetrates down through the {\df inner} mantles to
what we call the {\df outer core}, what remains when all
outer layers of forcing have been stripped away. Many
fundamental questions remain open.
\end{abstract}

\title{Set-theoretic geology}\maketitle

The technique of forcing in set theory is customarily
thought of as a method for constructing {\it outer} as
opposed to {\it inner} models of set theory. A set theorist
typically has a model of set theory $V$ and constructs a
larger model $V[G]$, the forcing extension, by adjoining a
$V$-generic filter $G$ over some partial order $\P\in V$. A
switch in perspective, however, allows us to view forcing
as a method of describing inner models as well. The idea is
simply to search inwardly for how the model $V$
might itself have arisen by forcing. Given a set-theoretic
universe $V$, we consider the classes $W$ over which $V$
can be realized as a forcing extension $V=W[G]$ by some
$W$-generic filter $G\of\P\in W$. This change in viewpoint
is the basis for a collection of questions leading to the
topic we refer to as set-theoretic geology. In this
article, we present some of the most interesting initial
results in the topic, along with an abundance of open
questions, many of which concern fundamental issues.

\section{The mantle}

We assume that the reader is familiar with the technique of
forcing in set theory. Working in \ZFC\ set theory and
sometimes in \GBC\ set theory, we suppose $V$ is the
universe of all sets. A class $W$ is a {\df ground} of $V$,
if $W$ is a transitive class model of \ZFC\ and $V$ is
obtained by set forcing over $W$, that is, if there is some
forcing notion $\P\in W$ and a $W$-generic filter $G\of\P$
such that $V=W[G]$. Laver
\cite{Laver2007:CertainVeryLargeCardinalsNotCreated} and independently Woodin
\cite{Woodin2004:RecentDevelopmentsOnCH}
\cite{Woodin2004:CHMultiverseOmegaConjecture} proved in
this case that $W$ is a definable class in $V$, using
parameters in $W$, a result that was generalized by Hamkins to include many natural instances of class forcing (see theorems~\ref{Theorem.UniformDefinabilityOfGrounds} and~\ref{Theorem.DefinabilityOfPseudoGrounds}). Building on these ideas, Hamkins and Reitz~\cite{Hamkins2005:TheGroundAxiom, Reitz2006:Dissertation,Reitz2007:TheGroundAxiom} introduced
the following axiom.

\begin{definition}\rm
The {\df ground axiom} \GA\ is the assertion that the
universe $V$ is not obtained by set forcing over any
strictly smaller ground model.
\end{definition}

Because of the quantification over classes, the ground
axiom assertion appears at first to be fundamentally
second order in nature, but Reitz
\cite{Reitz2007:TheGroundAxiom, Reitz2006:Dissertation}
proved that it is actually first-order expressible (an
equivalent claim is implicit, independently, in
\cite{Woodin2004:CHMultiverseOmegaConjecture}).

\begin{definition}\rm
A class $W$ is a {\df bedrock} for $V$ if it is a ground of
$V$ and minimal with respect to the forcing-extension
relation.
\end{definition}

Since a ground of a ground is a ground, we may equivalently
define that $W$ is a bedrock of $V$ if it is a ground of
$V$ and satisfies the ground axiom. Also, since by fact
\ref{Fact.IntermediateModelsAreGrounds} any inner model $U$
of \ZFC\ with $W\of U\of V$ for some ground $W$ of $V$ is
both a forcing extension of $W$ and a ground of $V$, we may
equivalently define that $W$ is a bedrock for $V$ if it is
a ground of $V$ that is minimal with respect to inclusion
among all grounds of $V$. It remains an open question
whether there can be a model $V$ having more than one
bedrock model.

In this article, we attempt to carry the investigation
deeper underground, bringing to light the structure of the
grounds of the set-theoretic universe $V$ and how they
relate to the grounds of the forcing extensions of $V$.
Continuing the geological metaphor, the principal new
concept is:

\begin{definition}\rm
The {\df mantle} $\Mantle$ of a model of set theory is the
intersection of all of its grounds.
\end{definition}

The ground axiom can be reformulated as the assertion
$V=\Mantle$, that is, as the assertion that $V$ is its own
mantle. The mantle was briefly mentioned, unnamed, at the
conclusion of~\cite{Reitz2007:TheGroundAxiom}, where the
question was raised whether it necessarily models \ZFC. Our
main theorem in this article is a converse of sorts:

\begin{maintheorem}
Every model of \ZFC\ is the mantle of another model of
\ZFC.
\end{maintheorem}

This theorem is a consequence of the more specific claims
of theorems~\ref{Theorem.V=M=gM=gHOD=HOD} and
\ref{Theorem.V=M=gM=gHOD,V[G]=HOD}, in which we are able
not only to control the mantle of the target model, but
also what we call the generic mantle, as well as the \HOD\
and generic \HOD. We begin by proving that the mantle,
although initially defined with second-order reference to
the meta-class of all ground models, actually admits a
first-order definition. This is a consequence of the
following fact, formulated in second-order set theory.

\begin{theorem}[The ground-model definability theorem]\label{Theorem.UniformDefinabilityOfGrounds}%
Every ground model $W$ of $V$ is definable in $V$ using a parameter from $W$. That is, there is a specific formula
$\phi(y,x)$ such that if $V=W[G]$ is a forcing extension of
a ground $W$ by $W$-generic filter $G\of\P\in W$, then
there is $r\in W$ such that
$$W=\{x\st\phi(r,x)\}.$$
\end{theorem}

The ground-model definability theorem was proved independently by Laver and Woodin, and proved by Hamkins in the strengthened form of theorem~\ref{Theorem.DefinabilityOfPseudoGrounds}, below, which extends the result beyond set-forcing to include also many natural instances of class forcing and non-forcing extensions. The history is that after Laver proved the ground-model definability theorem, he had noticed a similarity in his original proof (never published) with the main argument of Hamkins and Woodin in~\cite{HamkinsWoodin2000:SmallForcing}, and had contacted Hamkins about it, who provided a proof based on the approximation and cover properties, establishing theorem~\ref{Theorem.DefinabilityOfPseudoGrounds}, and it is this proof that Laver adopted in~\cite{Laver2007:CertainVeryLargeCardinalsNotCreated}.

\begin{theorem}\label{Theorem.DefinabilityOfPseudoGrounds}
Every pseudo-ground model $W$ of $V$ is definable in $V$ using a parameter from $W$. That is, if an inner model $W\of V$ of \ZFC\ satisfies the $\delta$-approximation and cover properties for some regular cardinal $\delta$ of $W$ for which also $(\delta^+)^W=(\delta^+)^V$, then $W$ is definable in $V$ using parameter $r=({}^\ltdelta 2)^W$.
\end{theorem}

A {\df pseudo ground} is simply an inner model $W\of V$ having the properties stated in the second part of the theorem, and by lemma~\ref{Lemma.ClosurePointForcing} this case includes every set-forcing extension, as well as the extensions arising with many instances of class forcing, such as the canonical forcing of the \GCH, the global Laver preparation and other progressively closed Easton-support iterations. In light of this, we should like briefly to mention and then leave for the future the idea of undertaking the entire set-theoretic geology project of this article in the more general context of pseudo-grounds, rather than only the set-forcing grounds, for this context would include these other natural extensions that are not a part of the current investigation. Because the basic definability theorems apply in this more general case, however, one may formalize the pseudo-ground analogues of the ground axiom, bedrocks, the mantle and the generic mantle.

\begin{definition}[\cite{Hamkins2003:ExtensionsWithApproximationAndCoverProperties}]
\label{Definition.deltacoverandapproximation} \rm Suppose
that $W\subseteq V$ are both transitive models of (a
suitable small fragment of) \ZFC\ and $\delta$ is a
cardinal in $V$.
\begin{enumerate}
 \item The extension $W\of V$ exhibits the {\df
     $\delta$-approximation property} if whenever $A\in
     V$ with $A\subseteq W$ and $A\cap B\in W$ for all
     $B\in W$ with $|B|^{W} < \delta$, then $A\in W$.
 \item The extension $W\of V$ exhibits the {\df
     $\delta$-cover property} if for every $A\in V$
     with $A\subseteq W$ and $|A|^{V} < \delta$, there
     is $B\in W$ such that $A\of B$ and $|B|^{W} <
     \delta$.
\end{enumerate}
\end{definition}

Such extensions are pervasive in set theory, in light of the following key lemma.

\begin{lemma}[{\cite[lemma~13]{Hamkins2003:ExtensionsWithApproximationAndCoverProperties}}]\label{Lemma.ClosurePointForcing}
If\/ $V\subset V[G]$ is a forcing extension by forcing of
the form $\Q_1*\dot\Q_2$, where $\Q_1$ is nontrivial and
$\forces_{\Q_1}$``\hspace{2 pt}$\dot\Q_2$ is ${\leq}|\Q_1|$
strategically closed'', then $V\of V[G]$ satisfies the
$\delta$-approximation and $\delta$-cover properties for
$\delta=|\Q_1|^{+}$.
\end{lemma}

In particular, since the second factor $\Q_2$ can be trivial, it follows that every forcing extension $V\of V[G]$, with $V$-generic filter $G\of\Q\in V$, exhibits the $\delta$-approximation and $\delta$-covering property for $\delta=|\Q|^+$, and this is why theorem~\ref{Theorem.DefinabilityOfPseudoGrounds} generalizes theorem~\ref{Theorem.UniformDefinabilityOfGrounds}.
Lemma~\ref{Lemma.ClosurePointForcing} has an improved proof in~\cite[lemma~12]{HamkinsJohnstone2010:IndestructibleStrongUnfoldability}, following Mitchell's treatment~\cite{Mitchell2003:ANoteOnHamkinsApproximationLemma}, which really goes all the way back to~\cite{Mitchell72:AronszajnTreesAndIndependenceProperty}, and see also the generalization in~\cite[\S1]{Unger:FragilityAndIndestructibilityII}.


The parameter in Laver's original proof of the ground-model definability theorem was a large initial segment $r=V_\theta^W$ of the ground model universe, well above the size of the forcing. This was reduced to parameter $r=P(\delta)^W$ in~\cite{Laver2007:CertainVeryLargeCardinalsNotCreated} by the Hamkins argument, where the forcing $\P$ has the $\delta$-approximation and cover property, and for this any cardinal $\delta\geq|\P|^+$ suffices by lemma~\ref{Lemma.ClosurePointForcing}. Hamkins and Johnstone subsequently noted that $P(\delta)^W$ is actually determined via the $\delta$-approximation property by $r=({}^\ltdelta 2)^W$, and so this smaller parameter suffices, as we have stated in the results here. If $\delta=\gamma^+$, then this parameter is equidefinable from $P(\gamma)^W$, and so in the case of a set forcing extension $V=W[G]$ with $G\of\P\in W$, we may define $W$ in $V$ with parameter $r=P(\P)^W$, which seems extremely natural. Nevertheless, for many instances of set forcing, theorem~\ref{Theorem.DefinabilityOfPseudoGrounds} shows that a much smaller parameter than $P(\P)^W$ suffices, since by lemma~\ref{Lemma.ClosurePointForcing} there are many natural forcing notions $\P$ that exhibit the $\delta$-approximation and cover properties even when $|\P|$ is far larger than $\delta$. For example, if one adds a Cohen real and then performs any countably strategically closed forcing, the original ground model $W$ will be definable in the final extension from parameter $P(\omega)^W$, even when the forcing is enormous.

The proof of theorem~\ref{Theorem.DefinabilityOfPseudoGrounds} is based fundamentally on the following lemma, of which we shall make use subsequently. It is convenient to express the lemma using the theory $\ZFC_\delta$, isolated by Reitz in~\cite{Reitz2006:Dissertation,Reitz2007:TheGroundAxiom}, which has the axioms of Zermelo set theory, the foundation axiom, the axiom of choice, the $\leqdelta$-replacement axioms (asserting all instances of replacement for functions with domain $\delta$, a fixed regular cardinal), together with the assertion that every set is coded by a set of ordinals. This theory is formalized in the language with a constant symbol for $\delta$, and it is an easy exercise to see that $V_\theta\satisfies\ZFC_\delta$ for any $\beth$-fixed point $\theta=\beth_\theta$ of cofinality larger than $\delta$, a regular cardinal.

\begin{lemma}[{Hamkins, see~\cite[lemma~7.2]{Reitz2007:TheGroundAxiom}}]\label{Lemma.ApproximationCoverImpliesW=W'}
Suppose that $W$, $W'$ and $U$ are transitive models of $\ZFC_\delta$, where $\delta$ is a fixed regular cardinal, that $W\of U$ and $W'\of U$ have the $\delta$-approximation and $\delta$-cover properties, that $({}^\ltdelta 2)^W=({}^\ltdelta 2)^{W'}$ and that $(\delta^+)^W=(\delta^+)^{W'}=(\delta^+)^U$. Then~$W=W'$.
\end{lemma}

In the context of theorem~\ref{Theorem.DefinabilityOfPseudoGrounds}, the precise definition of the ground model $W$ in its forcing extension $V=W[G]$, using parameter $r=({}^\ltdelta 2)^W$, is that $x\in W$ just in case, in $V$, there is a cardinal $\theta=\beth_\theta$ of cofinality larger than $\delta$ and a subset $M\of V_\theta$ exhibiting the $\delta$-approximation and $\delta$-cover properties, with $x\in M\satisfies\ZFC_\delta$ and $r=({}^\ltdelta 2)^M$ and $(\delta^+)^M=\delta^+$. Lemma~\ref{Lemma.ApproximationCoverImpliesW=W'} shows that any such $M$ will have to be $V_\theta^W$, since $V_\theta^W$ does have these properties. The results of~\cite{BagariaHamkinsTsaprounisUsuba:SuperstrongAndOtherLargeCardinalsAreNeverLaverIndestructible} rely in part on an analysis of the precise logical complexity of this definition.

For the subsequent theorems, we will make use of the
following folklore facts. An {\df inner model} is a
transitive class containing all the ordinals and satisfying
\ZF. We shall also frequently speak of inner models of
\ZFC\ as well.

\begin{fact}
If\/ $W=\set{\<r,x>\mid \phi(r,x)}$ is a definable class,
then the class of all $r$ for which the section
$W_r=\set{x\mid \<r,x>\in W}$ is an inner model of \ZFC\ is
definable.\label{Fact.IM}
\end{fact}

\begin{proof}This amounts to a standard result in Jech
\cite[Theorem 13.9]{Jech:SetTheory3rdEdition}, asserting
that the question of whether a transitive proper class
containing all the ordinals satisfies full \ZFC\ is
expressible by a single first-order formula. What needs to
be said is that $W_r$ is transitive, contains all the
ordinals, is closed under the \Godel\ operations, is almost
universal, meaning that every $A\of W_r$ is covered by some
$B\in W_r$, and satisfies the axiom of choice.
\end{proof}

\begin{fact}[{\cite[Corollary
15.43]{Jech:SetTheory3rdEdition}}] Suppose that $V\of V[G]$
is the forcing extension arising from a $V$-generic filter
$G\of\P\in V$. If $U$ is a transitive model of \ZFC\ with
$V\of U\of V[G]$, then $U$ is a forcing extension of $V$
and a ground of $V[G]$. In fact, $U$ is a forcing extension
of $V$ by a filter $G_0$ on a complete subalgebra $\B_0$ of
the complete Boolean algebra
$\B=\RO(\P)$.\label{Fact.IntermediateModelsAreGrounds}
\end{fact}

\begin{proof}  We first show a general fact about forcing
extensions:  If $A \of \kappa$ is a set of ordinals in
$V[G]$, then there is a complete subalgebra $\D \subset \B$
such that $V[A] = V[\D \cap G]$. To see this, fix a name
$\dot{A}$ for $A$ and consider the set $X_A$ of boolean
values $\boolval{\alpha \in \dot{A}}$ for $\alpha < \kappa$
determining membership in $A$.  Taking $\D \subset \B$ to
be the complete subalgebra generated by $X_A$ in $V$, it is
straightforward to verify that $V[A] = V[\D \cap G]$ (here
we use the fact that $\D \cap G$ is determined by $X_A \cap
G$, which follows from the fact that $X_A$ generates $\D$).

Now consider our models $V \subset U \subset V[G]$.  Let $A
\subset \kappa$ be a set of ordinals in $U$ coding
$P(\B)^U$.  Let $\B_0$ be the complete subalgebra obtained
by applying the general fact given above to the set $A$,
and let $G_0 = \B_0 \cap G$ be the corresponding generic
filter.  It remains to argue that $V[G_0] = U$.  For the
forward inclusion, note that $A\in U$ and so $V[G_0] =
V[A]$ is contained in $U$.  For the reverse inclusion, it
suffices to show that every set of ordinals in $U$ is in
$V[G_0]$.  But this also follows from the general fact
above:  each such set of ordinals $B$ is determined (over
$V$) by $\D_B \cap G$ for an appropriate subalgebra $\D_B$.
However, $\D_B \cap G$ is in $P(\B)^U$, and therefore
appears in $A$. It follows that $B \in V[A]$, and so
$V[G_0] = U$.  The remaining claim, that $U$ is a ground of
$V$, relies on a general fact about complete embeddings of
partial orders:  If $\B_0 \subset \B$ is a complete
subalgebra, then $\B$ is forcing equivalent to the
iteration $\B_0 *\B/\dot{G_0}$, where the second factor is
(a name for) the quotient of $\B$ by the generic filter
$\dot{G_0}$ on $\B_0$.  The details of this relationship
are explored in~\cite[pp243--244, Ex.
D3--D4]{Kunen:Independence}.\end{proof}

The following theorem summarizes several senses in which we
have first-order definable access to the family of ground
models of the universe.

\begin{theorem}\label{Theorem.ParameterizedGroundsW_r}
There is a parameterized family $\set{W_r\mid r\in V}$ of
classes such that
\begin{enumerate}
 \item Every $W_r$ is a ground of $V$, and $r\in W_r$.
 \item Every ground of $V$ is $W_r$ for some $r$.
 \item The classes $W_r$ are uniformly definable in the
     sense that $\set{\<r,x>\mid x\in W_r}$ is first
     order definable without parameters.
 \item The relation ``$V=W_r[G]$, where $G\of\P\in W_r$
     is $W_r$-generic'' is first-order expressible in
     $V$ in the arguments $(r,G,\P)$.
 \item The definition relativizes down from $V$ in the
     sense that if $W_r\of U\of V$ for an inner model
     $U\satisfies\ZFC$, then $W_r^U=W_r$.
 \item The definition relativizes up from $V$ in the
     sense that for any $r$ and any forcing extension
     $V[G]$, there is $s$ with $W_r=W_s=W_s^{V[G]}$.
\end{enumerate}
\end{theorem}

\begin{proof}By theorem~\ref{Theorem.UniformDefinabilityOfGrounds},
there is a formula $\phi(r,x)$ such that every ground $W$
of $V$ has the form $W_r=\set{x\mid \phi(r,x)}$ for some
$r\in W$. Furthermore, by fact~\ref{Fact.IM}, the question
of whether for a given $r$ the class $U_r=\set{x\mid
\phi(r,x)}$ defines a transitive inner model of \ZFC\ is
first-order expressible. Thus, by quantifying over the
possible partial orders $\P$ in this class and the possible
$U_r$-generic filters $G\of\P$, we can express in a
first-order manner whether $U_r$ is in fact a ground of $V$
or not. Thus, we define $W_r=\set{x\mid \phi(r,x)}$, if
this is a ground of $V$, and otherwise $W_r=V$. It follows
that $W_r$ is defined for all $r$, that $r\in W_r$, that
the ground models of $V$ are exactly the classes $W_r$ and
that the relation ``$x\in W_r$'' is first-order definable.
So (1), (2) and (3) hold. Using this, we conclude that the
relation of (4) is also first-order definable, since we
merely express that $\P$ is a partial order in $W_r$, that
$G$ is a filter meeting every dense subset of $\P$ in $W_r$
and that every set is the interpretation of a $\P$-name in
$W_r$ by $G$. For (5), suppose that $W_r\of U\of V$ and
$U\satisfies\ZFC$. Since $W_r$ is a ground of $V$, we may
exhibit it as $V = W_r[G]$, where $G\of \P$ is
$W_r$-generic. Furthermore, by the details of the
definition $\phi$ provided by theorem~\ref{Theorem.DefinabilityOfPseudoGrounds} and fact~\ref{Fact.IM}, we may assume
that $r = ({}^\ltdelta 2)^{W_r}$, where $\delta>|\ro{\P}|$. Since
$W_r\of U\of W_r[G] = V$, it follows by fact
\ref{Fact.IntermediateModelsAreGrounds} that $U = W_r[G_0]$
is a forcing extension of $W_r$ by a complete subalgebra of
$\ro{\P}$. Since that complete subalgebra also has size
less than $\delta$, we may use the same parameter $r =
({}^\ltdelta 2)^{W_r}$ to define $W_r$ inside $U$. Thus, $W_r^V =
W_r^U$, as desired for (5). For (6), because $W_r$ is a
ground of $V$, it is also a ground of $V[G]$, and so $W_r =
W_s^{V[G]}$ for some $s$. By (5), since $V$ is intermediate
between $W_s^{V[G]}$ and $V[G]$, we have $W_s^V =
W_s^{V[G]}$, establishing (6).\end{proof}

One may also prove a pseudo-ground analogue of theorem~\ref{Theorem.ParameterizedGroundsW_r} with essentially the same arguments (see~\cite{BagariaHamkinsTsaprounisUsuba:SuperstrongAndOtherLargeCardinalsAreNeverLaverIndestructible}). We shall now fix the notation $W_r$ of theorem~\ref{Theorem.ParameterizedGroundsW_r} for the rest of this
article, so that $W_r$ now refers to the ground of $V$
defined by parameter $r$, and this definition is completely
uniform, so that the classes $W_r^U$ index the grounds of
$U$ for any model  $U$ of \ZFC, as $r$ ranges over $U$.

Theorem~\ref{Theorem.ParameterizedGroundsW_r} shows that we
are essentially able to treat the meta-class collection of
all ground models of $V$ in an entirely first-order manner.
Quantifying over the ground models becomes simply
quantifying over the parameters $r$ used to define them as
$W_r$, and the basic questions about the existence and
relations among various ground models become first-order
properties of these parameters. For example, the ground
axiom asserts $\forall r\, V=W_r$, and a ground model $W_r$
is a bedrock model if $\forall s\,(W_s\of W_r\to W_s=W_r)$.
Alternatively, one could also have defined that $W_r$ is a
bedrock if $(\GA)^{W_r}$, that is, if the ground axiom
holds when relativized to $W_r$.

\begin{corollary}\label{Corollary.MantleIsDefinable}
The mantle of any model of set theory is a parameter-free
uniformly first-order-definable transitive class in that
model, containing all ordinals.
\end{corollary}

\begin{proof} We've already done the work above.
If $\Mantle$ is the mantle of $V$, then
$$\Mantle=\Intersect_r W_r,$$ or in other words,
$x\in\Mantle\iff \forall r\, x\in W_r$, which is a uniform
parameter-free definition. And being the intersection of
grounds, the mantle is clearly transitive and contains all
ordinals.\end{proof}

\begin{definition}\rm
The ground models of the universe $V$ are {\df downward
directed} under inclusion if the intersection of any two
contains a third, that is, if for every $r$ and $s$, there
is $t$ with $W_t\of W_r\intersect W_s$. The ground models
are {\df locally} downward directed if for every $r$ and
$s$ and every set $B$, there is $t$ with $W_t\intersect
B\of W_r\intersect W_s$. And similarly there are the upward-directed notions.
\end{definition}

Numerous fundamental open questions arise once one is
situated in the context of the collection of all ground
models.

\begin{question}[\cite{Reitz2006:Dissertation,Reitz2007:TheGroundAxiom}]
Is there a unique bedrock model when one
exists?\label{Question.UniqueBedrock?}
\end{question}

Reitz
\cite{Reitz2006:Dissertation,Reitz2007:TheGroundAxiom}
shows that there can be models of \ZFC\ having no bedrock
models at all. A more general question is:

\begin{question}[\cite{Reitz2006:Dissertation}] Are the ground
models of the universe downward
directed?\label{Question.GroundsDirected?}
\end{question}

If the ground models are downward directed, then of course
there cannot be distinct bedrock models, since the bedrock
models are exactly the minimal ground models with respect
to inclusion. So an affirmative answer to question
\ref{Question.GroundsDirected?} implies an affirmative
answer to question~\ref{Question.UniqueBedrock?}. Allowing
larger intersections, let us say that the ground models of
the universe are downward {\df set-directed} under
inclusion if the intersection of any set-indexed collection
of ground models contains a ground model, that is, if for
every set $A$ there is $t$ such that
$W_t\of\Intersect_{r\in A} W_r$. Relaxing this somewhat, we
say that the ground models are {\df locally} downward
set-directed if for any sets $A$ and $B$ there is $t$ such
that $W_t\intersect B\of \Intersect_{r\in A} W_r$.

\begin{question}
Are the ground models of the universe downward
set-directed? Are they locally downward set-directed? Are
they locally downward
directed?\label{Question.GroundsSetDirected?}
\end{question}

Affirmative answers to all of the questions we have asked
so far are consistent with \ZFC. The answers are trivially
affirmative under the ground axiom, for example, since in
this case the only ground model is the universe itself.
More generally, the questions all have affirmative answers
if there is a least member of the collection of grounds
under inclusion, or equivalently, in other words, if the
mantle itself is a ground. This last hypothesis is true in
any set-forcing extension of $L$.

\begin{definition}\rm
An inner model $W$ is a {\df solid bedrock} for $V$ if it
is a ground of $V$ and contained in all other grounds of
$V$. The {\df solid bedrock axiom} is the assertion that
there is a solid bedrock.
\end{definition}

The bedrock axiom of~\cite{Reitz2007:TheGroundAxiom}, in
contrast, asserts merely that there is a bedrock, but not
necessarily that this bedrock is solid. In other words, the
bedrock axiom asserts that there is a minimal ground, while
the solid bedrock axiom asserts that there is a smallest
ground.

\goodbreak
\begin{definition}\rm The
{\df downward-directed grounds} hypothesis \DDG\ is the
assertion that the ground models of the universe are
downward directed. The {\df strong} \DDG\ is the assertion
that the ground models of the universe are downward
set-directed. The {\df generic} \DDG\ is the assertion that
the \DDG\ holds in all forcing extensions. The {\df generic
strong} \DDG\ is the assertion that the strong \DDG\ holds
in all forcing extensions.
\end{definition}

If $W$ is a solid bedrock, it follows of course that $W$ is
precisely the mantle. Thus, the solid bedrock axiom is
equivalent to the assertion that the mantle is a ground,
and it immediately implies the strong \DDG\ and all of the
directedness properties that we have considered. Reitz
observed in~\cite{Reitz2006:Dissertation} that under the
continuum coding axiom \CCA, the assertion that every set
is coded into the \GCH\ pattern, it follows that the
universe is contained in every ground of every forcing
extension. Thus, \CCA\ implies that the universe is a solid
bedrock in all its forcing extensions, and so the strong
\DDG\ holds in $V$ and all forcing extensions. That is, the
\CCA\ implies the generic strong \DDG\ hypothesis. An
affirmative answer to (the first part of) question
\ref{Question.GroundsSetDirected?} implies an affirmative
answer to question~\ref{Question.GroundsDirected?}, which
we have said implies an affirmative answer to question
\ref{Question.UniqueBedrock?}.

Let us now prove that we can obtain affirmative answers to
all of the above questions in the case where the universe
is constructible from a set. Define that a {\df generic
ground} of $V$ is a ground of a forcing extension of $V$.

\begin{theorem}\label{Theorem.IfConstructibleFromASetThenSetDirectedGrounds}
If\/ $V=L[a]$ for a set $a$, then the generic strong \DDG\
holds. Indeed, all the models of the form
$$H^\alpha=\HOD^{\V^{\Coll(\omega,\alpha)}}$$
are grounds of $V$ and they are downward set-directed and
dense in the grounds and indeed in the generic grounds, in
the sense that every ground $W$ of $V$ or of any forcing
extension $V[G]$ has $H^\alpha\of W$ for all sufficiently
large $\alpha$. These grounds therefore exhibit the desired
downward set-directedness.
\end{theorem}

\begin{proof}
For any ordinal $\alpha$, we let $H^\alpha$ denote the
class $\HOD^{\V^{\Coll(\omega,\alpha)}}$ consisting of all
$x$ forced by $\one$ to be in the $\HOD$ of the extension via $\Coll(\omega,\alpha)$.
The second part of the theorem can be formalized as the
statement that for any poset $\P$ and any $\P$-name
$\dot{r}$ we have that $\P$ forces $\check H^\alpha\of
W_{\dot{r}}^{V[\dot G]}$ for all sufficiently large
$\alpha$.

To begin the proof, note that if $G$ is
$\Coll(\omega,\alpha)$-generic over $V$ for some ordinal
$\alpha$, then since the forcing is weakly homogeneous and
ordinal-definable, it follows that $\HOD^{V[G]}\of\HOD^V$
and moreover $x\in\HOD^{V[G]}$ if and only if
$\one\forces\check x\in\HOD$, which implies that
$\HOD^{V[G]}$ does not depend on $G$ and is equal to
$H^\alpha$. In particular, $H^\alpha$ is an inner model of
$V$ that satisfies \ZFC. Similar homogeneity reasoning
shows that $H^\beta\of H^\alpha$ whenever $\alpha<\beta$,
and so the collection of $H^\alpha$ is downward
set-directed.

If $W$ is a ground of $\V$, so that $V=W[g]$ for some
$W$-generic $g\of\P\in W$, then let $\theta$ be any
cardinal at least as large as $\card({\P})$ and suppose
that $G$ is $V$-generic for $\Coll(\omega,\theta)$. We may
absorb the $\P$ forcing into the collapse, since
$\P\times\Coll(\omega,\theta)$ is forcing equivalent to
$\Coll(\omega,\theta)$ by~\cite[Lemma
26.7]{Jech:SetTheory3rdEdition}, and so
$$V[G]=W[g][G]=W[G']$$ for some $W$-generic filter $G'\of
\Coll(\omega,\theta)$. Furthermore, we may observe that
$$\HOD^{\V[G]}=\HOD^{W[G']}\subseteq W,$$
where the inclusion holds because $\Coll(\omega,\theta)$ is
weakly homogeneous and ordinal definable in $W$. Finally,
using our hypothesis that $V=L[a]$, we argue that
$\HOD^{V[G]}$ is a ground of $V$. This is a consequence of
\Vopenka's theorem~\cite[Theorem
15.46]{Jech:SetTheory3rdEdition}, which asserts that every
set is generic over $\HOD$. In our case, since
$V[G]=L[a,G]$, we have that $a$ and $G$ are generic over
$\HOD^{V[G]}$, and so $\HOD^{V[G]}\of V\of
\HOD^{V[G]}[a,G]=V[G]$, trapping $V$ between $\HOD^{V[G]}$
and its forcing extension $V[G]$. It therefore follows by
fact~\ref{Fact.IntermediateModelsAreGrounds} that
$\HOD^{V[G]}$ is a ground of $V$, as desired. The downward
set-directedness of the grounds is now an obvious
consequence.

Finally, to see that the models $H^\alpha$ are dense in the
generic grounds, let $W$ be a ground of $V[h]$, so that
$W[g]=V[h]$ for some $W$-generic filter $g$ and $V$-generic
$h\of\Q\in V$. By applying what we have proved so far, but
in $V[h]$, we know that
$\HOD^{V[h]^{\Coll(\omega,\alpha)}}\of W$ for sufficiently
large $\alpha$. But by insisting that $\alpha$ is also
larger than $\Q$, the forcing $\Q$ is absorbed into the
collapse via
$\Q*\Coll(\omega,\alpha)\equiv\Coll(\omega,\alpha)$, and
from this it follows that
$\HOD^{V[h]^{\Coll(\omega,\alpha)}}=
\HOD^{V^{\Coll(\omega,\alpha)}}=H^\alpha$, and so
$H^\alpha\of W$ as desired. Lastly, again it is easy to see
as a consequence that the generic grounds are downward
set-directed, and so the generic \DDG\ holds in $V$.
\end{proof}

To what extent generally can we expect the grounds to be
directed, locally set directed or even fully set-directed?
A model of \ZFC\ having two bedrock models, providing a
negative answer to question~\ref{Question.UniqueBedrock?},
would be very interesting. Aside from the naturalness of
these hypotheses, their merit is that if they hold, the
mantle is well behaved. Before examining their effect on
the mantle, let's step back a little and look at a more
general setup. We say that $\{U_p\st p\in I\}$ is a {\df
parameterized family} of inner models if each $U_p$ is an
inner model and there is a formula $\varphi$ such that $x\in
U_p$ if and only if $\varphi(x,p)$ holds.

\begin{lemma}\label{Lemma.IntersectionsOfInnerModels}
If $\mathcal{W}=\{U_p\st p\in I\}$ is a parameterized
family of inner models, then $\bigcap\mathcal{W}$ is an
inner model if and only if
$V_\alpha\cap(\bigcap\mathcal{W})\in\bigcap\mathcal{W}$ for
every ordinal $\alpha$. Furthermore, this is indeed the
case if\/ $\bigcap\mathcal{W}$ is a definable class in
every $U\in\mathcal{W}$.
\end{lemma}

\begin{proof}
For the forward direction of the claimed equivalence,
assume that $\bigcap\mathcal{W}$ is an inner model, that
is, a transitive class model of \ZF\ containing all
ordinals. Given an ordinal $\alpha$, it follows that
$(V_\alpha)^{\bigcap\mathcal{W}}\in\Intersect\mathcal{W}$.
By absoluteness of rank,
$(V_\alpha)^{\bigcap\mathcal{W}}=V_\alpha\cap(\bigcap\mathcal{W})$,
which proves the forward direction. For the reverse
direction, note that $\bigcap\mathcal{W}$ is transitive,
contains all the ordinals, and is closed under the eight
G\"odel operations, since this is true of every member of
$\mathcal{W}$, as each member of $\mathcal{W}$ is an inner
model. So by fact~\ref{Fact.IM}, in order to prove that
$\bigcap\mathcal{W}$ is an inner model, all that's left to
show is that it is almost universal. So let $A$ be a subset
of $\bigcap\mathcal{W}$ and let $\alpha$ be the rank of
$A$. Then $A\of
V_\alpha\cap(\bigcap\mathcal{W})\in\bigcap\mathcal{W}$ by
assumption, showing almost universality, and hence proving
the equivalence.

Now assume that $\bigcap\mathcal{W}$ is a definable
subclass of every $U\in\mathcal{W}$. For any ordinal
$\alpha$ and any $U\in\mathcal{W}$, we have
$V_\alpha\cap(\bigcap\mathcal{W})=V_\alpha\cap
U\cap(\bigcap\mathcal{W})=V_\alpha^U\cap(\bigcap\mathcal{W})\in
U$, by the separation axiom in $U$. Since $U$ was
arbitrary, it follows that
$V_\alpha\cap(\bigcap\mathcal{W})\in\bigcap\mathcal{W}$, as
desired.
\end{proof}

Let's now return to the mantle and explore the effects our
concepts of directedness have on it.

\goodbreak
\begin{theorem}\ \label{Theorem.IfDirectedMantleHasZF}
\begin{enumerate}
\item If the \DDG\ holds, that is, if the grounds are
    downward directed,
  then the mantle is constant across these ground models.%
\label{item:IfDDThenMantleIsConstant}%
\item If the mantle is constant across the grounds,
    then it is a model of \ZF.%
\label{item:IfMantleIsConstantThenZF}%
\item If the strong \DDG\ holds, that is, if the
    grounds are downward set-directed,
     then the mantle is a model of \ZFC.%
\item Indeed, this latter conclusion can be made if the
    grounds are merely downward directed and locally
    downward set-directed.
\end{enumerate}
\end{theorem}

\begin{proof}  (1) Suppose that the ground models of $V$ are
downward directed. Fix a ground $W$. Since any ground of
$W$ is also a ground of $V$, it follows that the mantle of
$V$ is contained in the mantle of $W$. Conversely, if $a$
is not in the mantle of $V$, then it is not in some ground
$W'$, and so it is not in $W\intersect W'$. By
directedness, however, there is a ground $\bar{W}\of
W\intersect W'$. Clearly, $\bar{W}$ is a ground of $W$, so
that the mantle of $W$ is contained in $\bar{W}$. But
$a\notin\bar{W}$, so $a$ is not in the mantle of $W$. So
the mantle is constant among the ground models of $V$.

(2) If the mantle is constant across the grounds, then in
particular, it is a definable subclass of every ground.
Since the mantle is the intersection of all the grounds, it
is an inner model by lemma
\ref{Lemma.IntersectionsOfInnerModels}, which we proved
precisely for this purpose.

For (3) and (4), suppose that the ground models are
downward directed and locally downward set-directed. By (1)
and (2), the mantle satisfies \ZF. For \ZFC, consider any
set $y$ in the mantle $\Mantle $. Every ground model $W$,
being a model of \ZFC, has various well orderings of $y$;
what we must show is that there is a well ordering of $y$
in common to all the grounds. For each well ordering $z$ of
$y$ that is not in the mantle, there is a ground model
$W_{r_z}$ with $z\notin W_{r_z}$. By local
set-directedness, the family $\set{W_{r_z}\st z\text{ is a
well order of }y\text{ with }z\notin \Mantle}$ has a ground
model $W$ with $W\intersect B\of W_{r_z}$ for all such
relations $z$, where $B$ is the set in $V$ of all relations
on $y$. That is, any relation $z$ excluded from the mantle
is excluded by reason of not being in a particular ground
model $W_{r_z}$, and consequently it is also excluded from
$W$. Any well ordering of $y$ in $W$, therefore, is in the
mantle, and since $W\satisfies\ZFC$, there are many such
well orderings. So $\Mantle \satisfies\ZFC$, as desired for
(3) and (4).\end{proof}

The argument of theorem~\ref{Theorem.IfDirectedMantleHasZF}
can be viewed as an instance of the following general
phenomenon.

\begin{definition}\label{Definition.LocallyDownwardSetDirectedInGeneral}
\rm A parameterized family $\mathcal{W}=\{N_i\st i\in I\}$
of transitive sets or classes, where $I$ is a class, is
{\df locally downward set-directed}, if for any set $B$ and
any set $J\of I$, there is an $i_0\in I$ such that
$B\intersect N_{i_0}\of B\intersect (\bigcap_{i\in J}N_i)$.
\end{definition}

Clearly, it is equivalent to require merely that
$B\intersect N_{i_0}\of \bigcap_{i\in J} N_i$.

\begin{corollary}\ \label{Corollary.IntersectionGivesZF}
\begin{enumerate}
 \item If $\mathcal{W}$ is a locally downward set-directed
parameterized family of inner models of \ZFC\ and
$V_\alpha\cap(\bigcap\mathcal{W})\in\bigcap\mathcal{W}$ for
every ordinal $\alpha$, then $\Intersect\mathcal{W}$ satisfies
\ZFC.
 \item If\/ $\mathcal{W}=\set{W_\alpha\st\alpha\in\ORD}$ is
any definable sequence of inner models that is descending
in the sense that $\alpha<\beta$ implies $W_\beta\of
W_\alpha$, then $\bigcap\mathcal{W}$ is an inner model, and
if every $W_\alpha$ satisfies \ZFC, then so does
$\bigcap\mathcal{W}$.
\end{enumerate}%
\end{corollary}

\begin{proof}
The first claim follows by lemma
\ref{Lemma.IntersectionsOfInnerModels} and the argument of
theorem~\ref{Theorem.IfDirectedMantleHasZF}. For the second
claim, suppose that  $\mathcal{W}$ is a descending sequence
of inner models. To see that $\Intersect\mathcal{W}$ is an
inner model, it suffices to show by lemma
\ref{Lemma.IntersectionsOfInnerModels} that
$V_\alpha\cap(\bigcap\mathcal{W})\in\bigcap\mathcal{W}$ for
every ordinal $\alpha$. Fix any ordinal $\alpha$. Since
$V_\alpha$ is a set, we may simply wait for all the
elements of $V_\alpha$ that will eventually fall out of the
$V_\alpha^{W_\beta}$ to have fallen out---simply apply the
replacement axiom to the falling-out times---and thereby
find an ordinal $\beta$ such that
$V_{\alpha}^{W_{\beta'}}=V_\alpha^{W_\beta}$ for all
$\beta'\ge\beta$. Then $V_\alpha^{W_\beta}=V_\alpha\cap
W_\beta=V_\alpha\cap(\bigcap\mathcal{W})$, so all that
needs to be shown is that $V_\alpha^{W_\beta}\in W_\gamma$,
for every $\gamma$. But this is clear for $\gamma<\beta$,
as $V_\alpha^{W_\beta}\in W_\beta\of W_\gamma$ in that
case. And for $\gamma\ge\beta$, we have that
$V_\alpha^{W_\beta}=V_\alpha^{W_\gamma}\in W_\gamma$. So
$\Intersect\mathcal{W}$ is an inner model. To see that this
model satisfies choice if every member of $\mathcal{W}$
does, it suffices to observe that $\mathcal{W}$ is downward
set-directed.\end{proof}

Let us remark that having a descending sequence of
set-sized length does not suffice for the conclusion, even
when the intersection of all of its members is definable in
each of the models. In
\cite{McAloon1974:OnTheSequenceHODn}, for example, a model
of \ZFC\ is produced in which the sequence $\<\HOD^n\st
n<\omega>$ is definable, where $\HOD^{n+1}=\HOD^{\HOD^n}$,
and where the intersection
$\HOD^\omega=\bigcap_{n<\omega}\HOD^n$ does not satisfy the
axiom of choice. Since the model produced there is
constructible from a set, it follows from
\cite{Grigorieff1975:IntermediateSubmodelsAndGenericExtensions}
that the sequence of iterated $\HOD$s is definable in every
iterated $\HOD$, and so $\HOD^\omega$ is a class in every
$\HOD^n$ there. So by lemma
\ref{Lemma.IntersectionsOfInnerModels}, we know that
$\HOD^\omega$ is an inner model in that case, but still not
a model of \ZFC. Note that by corollary
\ref{Corollary.IntersectionGivesZF}, if the sequence
$\<\HOD^\alpha\st\alpha\in\ORD>$ is definable, then the
intersection
$\HOD^{\Ord}=\bigcap_{\alpha<\Ord}\HOD^\alpha$ is a model
of \ZFC, as every $\HOD^{\alpha+1}$ is a model of \ZFC\ and
$\HOD^{\Ord}=\bigcap_{\alpha<\Ord}\HOD^{\alpha+1}$.
Another natural example is given by
\cite{Dehornoy1978:IteratedUltrapowersAndPrikyForcing},
where it is shown that if $\kappa$ is a measurable cardinal
and $\mu$ is a normal measure on $\kappa$, then the
intersection of, for example, the first $\omega^2$ many
iterated ultrapowers of $\V$ by $\mu$ fails to satisfy the
axiom of choice. Clearly, though, this intersection is
definable in each of the earlier iterates, and so by lemma
\ref{Lemma.IntersectionsOfInnerModels} it is an inner
model. There is also an interesting example, given by a
1974 result due to Harrington appearing in~\cite[section
7]{Zadrozny1983:IteratingOrdinalDefinability}, where a
model is constructed in which the intersection
$\bigcap_{n<\omega}\HOD^n$ is not a model of \ZF, and
actually, neither the sequence $\<\HOD^n\st n<\omega>$ nor
$\bigcap_{n<\omega}\HOD^n$ is a class in that model.

It follows by fact~\ref{Fact.IntermediateModelsAreGrounds}
that if the mantle $M$ is a ground of $V$, then there are
only a set number of possible intermediate models $M\of
W\of V$, since each is determined by its filter on a
certain subalgebra of the forcing from $M$ to $V$. Thus,
under the solid bedrock axiom, there are only set many
grounds, in the sense that there is a set $I$ such that
every ground of $V$ is $W_r$ for some $r\in I$. We don't
know if the converse holds.

\begin{question}
Is the solid bedrock axiom equivalent to the assertion that
there are only set many grounds?
\label{Question.SBAequivalentToSetManyGrounds}
\end{question}

A strong counterexample to this would be a model $V$ having
only set many grounds, but no minimal ground. Any such
model would of course also be a counterexample to downward
set-directedness and the generic strong \DDG\ hypothesis.

\begin{theorem}
If the universe is constructible from a set, then question
\ref{Question.SBAequivalentToSetManyGrounds} has a positive
answer. Indeed, if the universe is constructible from a
set, then the following are equivalent:
\begin{enumerate}
\item There are only set many grounds.%
\item The bedrock axiom.%
\item The solid bedrock axiom.%
\end{enumerate}\label{Theorem:IfConstructibleFromASetThenSBAEquivToSetManyGrounds}
\end{theorem}

\begin{proof}
Suppose $\V$ is constructible from a set. For every ordinal
$\alpha$, let
$H^\alpha=\HOD^{\V^{\Coll(\omega,\alpha)}}\!\!$, as in the
proof of theorem
\ref{Theorem.IfConstructibleFromASetThenSetDirectedGrounds}.

(1$\implies$2) Assume that $\V$ has only set many grounds
$\set{W_r\st r\in I}$, for some set $I$. By theorem
\ref{Theorem.IfConstructibleFromASetThenSetDirectedGrounds},
for each $r\in I$ there is an ordinal $\alpha_r$ for which
$H^{\alpha_r}\of W_r$. Thus, if $\alpha=\sup_r\alpha_r$, we
have the ground $H^\alpha$ contained in all grounds $W_r$.
In other words, $H^\alpha$ is a smallest ground, verifying
both the bedrock axiom, as well as the solid bedrock axiom.

(2$\implies$3) Suppose that $W$ is a minimal ground. By
theorem
\ref{Theorem.IfConstructibleFromASetThenSetDirectedGrounds},
it follows that $W$ must have the form $H^\alpha$ for some
$\alpha$, and moreover that $W=H^\beta$ for all larger
$\beta$, since the $H^\gamma$ sequence is non-increasing
with respect to inclusion. Thus, $W$ is contained in every
$H^\gamma$, and since these models are dense in the
grounds, it follows that $W$ is contained in every ground.
So $W$ is a solid bedrock, as desired.

(3$\implies$1) This implication was proved in the remarks
before the statement of question
\ref{Question.SBAequivalentToSetManyGrounds}.
\end{proof}

\begin{theorem}
Every model $V$ of \ZFC\ has a class forcing extension
$V[G]$ in which the grounds are downward set-directed, but
there is no bedrock.\label{Theorem.DirectedButNoBedrock}
\end{theorem}

\begin{proof}
This is the essence of~\cite[Theorem
24]{Reitz2006:Dissertation}, where Reitz showed that every
model has a class forcing extension having no bedrock.
Beginning with any model $V$, we first move to a class
forcing extension $\Vbar$ exhibiting the continuum coding
axiom (\CCA), which asserts that every set of ordinals is
coded into the \GCH\ pattern of the continuum function $\kappa \mapsto 2^\kappa$, by
iteratively coding sets into this pattern. Then, in
$\Vbar$, let $\P$ be the Easton-support class product of
forcing adding a Cohen subset to every regular cardinal
$\lambda$ for which $2^\ltlambda=\lambda$, and suppose that
$\Vbar[G]$ is the corresponding class forcing extension.
This forcing preserves all cardinals, cofinalities and the
\GCH\ pattern over $\Vbar$. Consequently, the sets of
$\Vbar$ remain coded unboundedly often into the \GCH\
pattern of $\Vbar[G]$. Suppose that $W$ is a ground of $\Vbar[G]$. Since $W$ and $\Vbar[G]$ have the same eventual \GCH\ pattern, it follows that the sets of $\Vbar$ are also all coded into the \GCH\ pattern of $W$, and so $\Vbar\of W$.
At each coordinate $\lambda$ at which we have performed forcing, we may consider the generic filter added on that coordinate as a subset $G(\lambda)\of\lambda$. In this way, we may view the (union of) the entire filter $G$ as a subset of the ordinal plane $G=\prod_\lambda G(\lambda)\of\Ord\times\Ord$. Let $\beta$ be a regular cardinal larger than the size of the forcing from $W$ to $\Vbar[G]$, so that $W\of\Vbar[G]$ satisfies the $\beta$-approximation property by lemma~\ref{Lemma.ClosurePointForcing}. Since $G(\beta)$ is a subset of $\beta$, with all initial segments in $\Vbar$ and hence in $W$, it follows immediately by the approximation property that $G(\beta)\in W$. A similar argument works with many coordinates at once. Namely, if $\beta<\xi$, let $G^{\beta,\xi}$ be the part of the generic filter added at coordinates between $\beta$ and $\xi$, so that we may view $G^{\beta,\xi}\of(\beta,\xi)\times\xi$. The corresponding forcing $\P^{\beta,\xi}$ is $\leqbeta$-closed in $\Vbar$. Suppose that $A\of(\beta,\xi)\times\Ord$ is a piece of the domain with $A\in W$ and of size less than $\beta$ in $W$. Since $A$ is also in $\Vbar[G]$, it follows by the closure of the forcing at coordinate $\beta$ and beyond that $A\in \Vbar[G_\beta]$, the extension with forcing only at coordinates strictly before $\beta$. Since that forcing has size at most $\beta$, it follows that $A$ is covered by a set $B\fo A$ with $B\in \Vbar$ and of size at most $\beta$ in $\Vbar$. In this case, since the forcing $\P^{\beta,\xi}$ is $\leqbeta$-closed in $\Vbar$, it follows that $G^{\beta,\xi}\intersect B\in\Vbar$ and consequently also $G^{\beta,\xi}\intersect B\in W$. By intersecting with $A$, we see that $G^{\beta,\xi}\intersect A\in W$ as well, and so we have established that every $\delta$-approximation to $G^{\beta,\xi}$ is in $W$. By the $\beta$-approximation property of $W\of\Vbar[G]$, therefore, it follows that $G^{\beta,\xi}\in W$. By considering $\xi$ arbitrarily large, this implies $\Vbar[G^\beta]\of W$, where $G^\beta$ is the generic filter at all coordinates $\beta$ and higher, and so $W$ is an intermediate \ZFC\ model between $\Vbar[G^\beta]\of W\of\Vbar[G]$. Since $\Vbar[G]$ arises via set forcing over $\Vbar[G^\beta]$ by forcing with the coordinates up to and including $\beta$, we conclude that $\Vbar[G^\beta]$ is a deeper ground than $W$. Such grounds are therefore dense below the grounds of $\Vbar[G]$ and form a strictly descending sequence of order type $\ORD$. It follows that the grounds are set-directed and have no minimal member, establishing the theorem.\end{proof}

The class forcing extension $\V[G]$ of the previous theorem
has the nice property that for every ordinal $\alpha$, we
have $\V_\alpha^{\V[G]}=\V_\alpha^{\V[G_\alpha]}$ for some
$G_\alpha$ that is set-generic over $\V$. So in a sense,
$\V[G]$ is ``close to $\V$''. Since the assumption that the
universe is constructible from a set considerably
simplifies the geology, the following is an interesting
strengthening of the previous theorem. The class forcing
extension constructed here will not be close to $\V$ in the
previous sense, though.

\begin{theorem}\label{Theorem.ConstructibleFromSetButNonBA}
Every model of set theory has a class forcing extension of
the form $L[r]$, where $r\sub\omega$, in which there is no
bedrock, but the grounds are downward set-directed.
\end{theorem}

\begin{proof}
First of all, the grounds of any model of the form $L[r]$
are downward set-directed by theorem
\ref{Theorem.IfConstructibleFromASetThenSetDirectedGrounds},
so the final claim of the theorem will come for free.

The desired model $L[r]$ is produced by Jensen coding, and
the content of our proof is the observation that this model
has no bedrock. In the Jensen construction (see
\cite{BellerJensenWelch1982:CodingTheUniverse} for a
standard reference), one passes first to a model $(W,A)$,
with $A\of\ORD$ specifically arranged so that
$H^W_\alpha=L_\alpha[A]$ whenever $\alpha$ is a cardinal of
$W$. One next extends $W$ by forcing with Jensen's class
partial order $\P$. Suppose that $G\of\P$ is $W$-generic,
and let $W'=W[G]$, which has the form $L[r]$ for some real
$r$ in $W'$. The forcing $G$ adds a class $D\of\ORD$ such
that $W[G]=L[D]$. If $\tau$ is a cardinal of $W$, then $\P$
splits as $\P_\tau*\P^{D_\tau}$. The forcing $\P_\tau$ is a
class forcing coding $\V$ as a set $D_\tau=D\cap\tau^+$, in
the sense that $W[G\restrict\P_\tau]=L[D_\tau]$, and
$\P^{D_\tau}$ is a set forcing, essentially an iteration of
almost disjoint coding, which brings the coding from
$\tau^\plus$ down to $\omega$ in the sense that it codes
$D_\tau$ as a subset $D_0$ of $\omega$. In particular,
$L[D_\tau]$ is a ground of $W'$. Furthermore, if
$\tau'>\tau$ is a larger cardinal of $W$, then $L[D_\tau]$
is a nontrivial set-forcing extension of $L[D_{\tau'}]$,
and so $W'$ has class-many grounds. Since $W'=L[r]$ is
constructible from a set, it follows from this and theorem
\ref{Theorem:IfConstructibleFromASetThenSBAEquivToSetManyGrounds}
that the bedrock axiom fails there.
\end{proof}

Our most fundamental lack of knowledge about the mantle is
that, without any additional hypotheses, we do not yet know
whether or not it is necessarily a model of \ZFC\ or even
of \ZF.

\begin{question}
Does the mantle necessarily satisfy \ZF? Or \ZFC?
\end{question}

One may of course restrict the concept of grounds to
certain classes of forcing notions. Thus, let us say that
$W$ is a $\sigma$-closed ground if $W$ is a ground such
that the universe is obtainable from $W$ by a
$\sigma$-closed forcing $\P$ (note that $\P$ is
$\sigma$-closed in $W$ if and only if it is so in $\V$). As
a reminder, let us mention that by definition, a ground
must be a model of \ZFC. The $\sigma$-closed mantle is
simply the intersection of all $\sigma$-closed grounds.
More generally, we define the $\Gamma$-Mantle, for any
definable class $\Gamma$ of forcing notions, to be the
intersection of all grounds $W$ for which $V$ is a forcing
extension of $W$ by a forcing notion in $\Gamma^W$. Theorem
\ref{Theorem:SigmaClosedMantleWithoutChoice} shows that the
$\sigma$-closed mantle need not be a model of \ZFC, even if
the universe is constructible from a set. We will make use
of the following folklore result, emphasizing that its
proof makes no use of the axiom of choice.

\begin{lemma}[{\cite[Lemma 2.5]{Solovay1970:AModelInWhichEverySetIsLebesgueMeasurable}}]
Assuming \ZF, if $G\cross H$ is $V$-generic for product
forcing $\P\cross\Q$, then $V[G]\intersect
V[H]=V$.\label{Lemma.V[G]intersectV[H]=V}
\end{lemma}

\begin{proof}
Certainly $V\of V[G]\intersect V[H]$. Conversely, suppose
that $x\in V[G]\intersect V[H]$. We may assume by
$\in$-induction that $x\of V$. By assumption, there is a
$\P$-name $\tau$ and a $\Q$-name $\sigma$ such that
$x=\tau_G=\sigma_H$. So there must be a condition $(p,q)\in
G\cross H$ forcing that $\tau_G=\sigma_H$. We claim that
$p$ already decides the elements of $\tau$. If not, we
could find $V[H]$-generic filters $G'$ and $G''$, both
containing $p$, such that $\tau_{G'}\neq\tau_{G''}$.  But
both must be equal to $\sigma_H$, since $G'\cross H$ and
$G''\cross H$ are both $V$-generic and contain $(p,q)$, a
contradiction. Thus, $p\forces\tau\in\check V$, and so
$x\in V$.\end{proof}

\begin{theorem}
If \ZFC\ is consistent, then so is the theory consisting of
the following axioms:
\begin{enumerate}
\item \ZFC,%
\item $\V=L[A]$, for a set $A\sub\omega_1$,%
\item the $\sigma$-closed mantle is the same as
    $L(\reals)$ and fails to satisfy the axiom of
    choice.
\end{enumerate}%
\label{Theorem:SigmaClosedMantleWithoutChoice}
\end{theorem}

\begin{proof}  Starting with a model of $\ZFC$, we can pass to a
forcing extension $M$ whose $L(\reals)$ fails to satisfy
the axiom of choice. For example, adding $\aleph_1$ Cohen
reals will do this; see~\cite[P.~245, ex.~(E3) and
(E4)]{Kunen:Independence}. The model $L(\reals)^M$ will
satisfy $\ZF+\DC$, though, since this always holds in the
$L(\reals)$ of a model of \ZFC. Now let $a,b$ be mutually
generic Cohen subsets of $\omega_1$ over $L(\reals)^M$.
Since the Cohen forcing at $\omega_1$ adds a well-ordering of $\R$, it follows
that $L(\reals)[a]$ is a model of \ZFC, and in fact
$L(\reals)[a]=L[a]$; and the same is true for
$L(\reals)[b]=L[b]$. Note that since $\DC$ holds in
$L(\reals)$, adding $a$ or $b$ adds no new
$\omega$-sequences, and hence doesn't change
$L(\reals)$.%
\footnote{It is an amusing observation of the first author
that over \ZF, the principle of dependent choices is in
fact equivalent to the statement that $\sigma$-closed
forcing adds no new countable sequences of ordinals.} So we
may unambiguously write $L(\reals)$ for
$L(\reals)^M=L(\reals)^{L(\reals)^M[a]}=L(\reals)^{L(\reals)^M[b]}$.
Consider now the model $N=L(\reals)[a,b]=L[a,b]$. Both
$L[a]$ and $L[b]$ are $\sigma$-closed grounds of $N$, and
so the $\sigma$-closed mantle of $N$ is contained in the
intersection of $L(\reals)[a]$ and $L(\reals)[b]$. But
since $a$ and $b$ are mutually generic, however, this
intersection is precisely $L(\reals)$ by lemma
\ref{Lemma.V[G]intersectV[H]=V}. Conversely, $L(\reals)$ is
contained in every $\sigma$-closed ground of $N$, since
every such ground must contain all $\omega$-sequences of
ordinals which are in $N$, and hence all the reals of $N$.
So the $\sigma$-closed mantle of $N$ is precisely
$L(\reals)^N$, which is the $L(\reals)$ we specifically
arranged at the beginning to violate the axiom of
choice.\end{proof}

The proof actually shows that every model $V$ of \ZFC\ has
a forcing extension $\Vbar$ in which there is a set
$A\of\omega_1$, such that $L[A]\satisfies\ZFC$, but the
$\sigma$-closed mantle of $L[A]$ is $L(\R)^\Vbar$ and fails
to satisfy the axiom of choice. Note also that statement
(2) of the previous theorem is optimal, in the sense that
if $\V=L[r]$, where $r\sub\omega$, then trivially $\V$ is
its own $\sigma$-closed mantle.

What we said about $\sigma$-closed geology remains true
when restricting to the $\sigma$-distributive notions. The
underlying concept in this context is that a ground $W$ is
a $\sigma$-distributive ground if there is a notion of
forcing $\P\in W$ such that $W$ thinks that $\P$ is
$\sigma$-distributive, and $\V$ is a forcing extension of
$W$ by $\P$. Note that $\P$ might fail to be
$\sigma$-distributive in $\V$ in that case.

\begin{question}
Under what circumstances is the mantle also a ground model
of the universe? That is, when does the solid bedrock axiom
hold?\label{Question.MantleIsAGround?}
\end{question}

\begin{question}
When does the universe remain a solid bedrock in all its
forcing
extensions?\label{Question.WhenDoesUniverseRemainStrongBedrock?}
\end{question}

We regard this latter property, the forcing invariance of
being a solid bedrock, as an inner-model-like structural
property, indicating when it holds that the universe is
highly regular or close to highly regular in some sense.
For example, the property holds in $L$, $L[0^\sharp]$,
$L[\mu]$ and many other canonical models, and it is also an
easy consequence of \CCA. Question
\ref{Question.WhenDoesUniverseRemainStrongBedrock?} has an
affirmative answer in exactly the models that are their own
generic mantle, a concept defined in section
\ref{Section.TheGenericMantle}.

Let us conclude this section with a brief combinatorial
analysis of the structure of the grounds of the universe.

\goodbreak\begin{theorem}\
\label{Theorem.LatticePropertiesOfGrounds}
\begin{enumerate}
 \item The collection of grounds between a fixed ground
     $W$ and the universe $V$ is an upper semi-lattice.
 \item If the grounds of $V$ are downward directed,
     then the grounds of $V$ are an upper semi-lattice.
 \item In this case, however, the upper semi-lattice of
     grounds of $V$ need not be a complete upper
     semi-lattice, even if the universe is
     constructible from a set.
 \item The grounds of $V$ need not be a lattice, even
     when the grounds are downward set-directed, and
     even if the universe is constructible from a set.
     There can be two grounds $W_0$ and $W_1$ whose
     intersection contains no largest model of \ZFC.
\end{enumerate}
\end{theorem}

\begin{proof}  (1) By fact~\ref{Fact.IntermediateModelsAreGrounds}, any \ZFC\ model
between $W$ and $V$ is a forcing extension $W[G]$ of $W$,
and a ground of $V$. Suppose that $W[G]$ and $W[H]$ are two
such extensions, with $G\of\P\in W$ and $H\of\Q\in W$ both
$W$-generic and $W[G],W[H]\of V$. Irrespective of mutual
genericity concerns, in $V$ we may form the object $G\cross
H$, a subset of $\P\cross\Q\in W$. Since $\P\cross\Q$ is
well-orderable in $W$, it follows that the model $W[G\cross
H]$, consisting of all sets constructible from elements of
$W$ and $G\cross H$, is the smallest model of \ZFC\
containing $W$ and $G\cross H$. This is a model of \ZFC\
between $W$ and $V$; it is a ground of $V$; and clearly it
contains both $W[G]$ and $W[H]$. Also, any other model of
\ZFC\ containing both $W[G]$ and $W[H]$ will contain $W$
and $G$ and $H$ and therefore also $G\cross H$, and so
$W[G\cross H]$ is the least upper bound of $W[G]$ and
$W[H]$, as desired.

(2) Suppose that the grounds of $V$ are directed. If $W_r$
and $W_s$ are two grounds, then there is some deeper ground
$W_t\of W_r\intersect W_s$. By applying (1) to the grounds
between $W_t$ and $V$, it follows that $W_r$ and $W_s$ have
a least upper bound.

For (3) and (4), we force over $L$ to construct a model
having the desired features for its grounds. Using Easton
forcing, let $\P$ be the forcing over $L$ to force
$2^{\aleph_n}=\aleph_{n+2}$ for all $n<\omega$. More
specifically, $\P$ is the full-support product of the
forcing $\Q_n=\Add(\aleph_n,\aleph_{n+2})$, as defined in
$L$. Suppose that $G\of\P$ is $L$-generic for this forcing.
Next, we add a Cohen real $c$ over $L[G]$. In the model
$V=L[G][c]$, we construct two grounds, so as to witness
(4). The first ground is simply $L[G]$. The second is
$L[G\restrict c]$, where $G\restrict c$ is the filter $G$
restricted to the coordinates $n$ with $n\in c$, viewing
$c$ as a subset of $\omega$. We claim that $L[G]$ and
$L[G\restrict c]$ have no greatest lower bound among the
grounds. Suppose towards contradiction that $W$, a ground
of $V$, is the greatest lower bound of $L[G]$ and
$L[G\restrict c]$. As all cardinals are preserved between
$L$ and $L[G][c]$, it follows that the models $W$, $L[G]$
and $L[G\restrict c]$ have all the same $\aleph_n$'s. Note
that if $n\in c$, then $L[G]$ and $L[G\restrict c]$ both
contain $L[G(n)]$, where $G(n)$ is the part of the generic
filter $G$ on coordinate $n$, to pump up the \GCH\ at
$\aleph_n$. Since $L[G(n)]$ is a ground of both $L[G]$ and
$L[G\restrict c]$ for such $n$, it follows by our
assumption on $W$ being a greatest lower bound that
$L[G(n)]\of W$ for any $n\in c$. This implies that the
\GCH\ fails in $W$ at every $\aleph_n$ for which $n\in c$.
If $n\notin c$, however, then the \GCH\ must hold in $W$ at
$\aleph_n$, since it holds in $L[G\restrict c]$ at these
$\aleph_n$, and cardinals are preserved. Thus, the \GCH\
pattern in $W$ on the $\aleph_n$'s exactly conforms with
the real $c$. But $c$ is not in $W$, because it is not in
$L[G]$, contradicting our assumption that $W$ satisfied
\ZFC. Thus, we have proved that the grounds $L[G]$ and
$L[G\restrict c]$ can have no greatest lower bound,
establishing (4). Statement (3) follows on general lattice
theoretic grounds, since if a downward directed upper
semi-lattice is complete with respect to joins, then for
any pair of nodes, the join of all lower bounds of these
nodes will be a greatest lower bound, and so it would be a
lattice. By (4), it need not be a lattice, and so we
conclude that (3), it need not be complete as an upper
semi-lattice.\end{proof}

In particular, theorem
\ref{Theorem.LatticePropertiesOfGrounds} statement (4)
provides an example of a model $L[G][c]$ having two grounds
$L[G]$ and $L[G\restrict c]$ whose intersection is not a
model of \ZF. Namely, the separation axiom fails in
$L[G]\intersect L[G\restrict c]$, since in this model one
can define the set of natural numbers $n$ for which the
power set of $\aleph_n$ is not bijective with
$\aleph_{n+1}$, but this set corresponds exactly to $c$,
which does not exist in $L[G]\intersect L[G\restrict c]$.

\section{A brief upward glance}

Set-theoretic geology naturally has a downward-oriented
focus, towards deeper grounds and mantles, but let us cast
a brief upward glance in this section by considering
several upwards closure issues. The second author heard the
following observation from Woodin in the early 1990s.

\begin{observation}\label{Observation:NonAmalgamableExtensions}
If $W$ is a countable model of \ZFC\ set theory, then there
are forcing extensions $W[c]$ and $W[d]$, both obtained by
adding a Cohen real, which are non-amalgamable in the sense
that there can be no model of \ZFC\ with the same ordinals
as $W$ containing both $W[c]$ and $W[d]$. Thus, the family
of forcing
extensions of $W$ is not upward directed.%
\end{observation}

\begin{proof}  Since $W$ is countable, let $z$ be a real coding
the entirety of $W$. Enumerate the dense subsets $\<D_n\st
n<\omega>$ of the Cohen forcing $\Add(\omega,1)$ in $W$. We
construct $c$ and $d$ in stages. We begin by letting $c_0$
be any element of $D_0$. Let $d_0$ consist of exactly as
many $0$s as $|c_0|$, followed by a $1$, followed by
$z(0)$, and then extended to an element of $D_0$.
Continuing, $c_{n+1}$ extends $c_n$ by adding $0$s until
the length of $d_n$, and then a $1$, and then extending
into $D_{n+1}$; and $d_{n+1}$ extends $d_n$ by adding $0$s
to the length of $c_{n+1}$, then a $1$, then $z(n)$, then
extending into $D_{n+1}$. Let $c=\union c_n$ and $d=\union
d_n$. Since we met all the dense sets in $W$, we know that
$c$ and $d$ are $W$-generic Cohen reals, and so we may form
the forcing extensions $W[c]$ and $W[d]$. But if $W\of
U\satisfies\ZFC$ and both $c$ and $d$ are in $U$, then in
$U$ we may reconstruct the map $n\mapsto\<c_n,d_n>$, by
giving attention to the blocks of $0$s in $c$ and $d$. From
this map, we may reconstruct $z$ in $U$, which reveals all
the ordinals of $W$ to be countable, a contradiction if $U$
and $W$ have the same ordinals.\end{proof}

Let us remark that many of the results in this section
concern forcing extensions of an arbitrary countable model
of set theory, which of course includes the case of
ill-founded models. Although there is no problem with
forcing extensions of ill-founded models, when properly
carried out, the reader may prefer to focus on the case of
countable transitive models for the results in this
section, and such a perspective will lose very little of
the point of our observations.

The method of observation
\ref{Observation:NonAmalgamableExtensions} is easily
generalized to produce three $W$-generic Cohen reals $c_0$,
$c_1$ and $c_2$, such that any two of them can be
amalgamated, but the three of them cannot. More generally:

\begin{observation}\label{Observation:NonAmalgamableExtensionsGeneralized}
If $W$ is a countable model of \ZFC\ set theory, then for any finite $n$ there are $W$-generic 
Cohen reals $c_0,c_1,\ldots,c_{n-1}$, such that any proper subset of them are mutually $W$-generic, so that 
one may form the generic extension $W[\vec c]$, provided that $\vec c$ omits at least one $c_i$, but there is no 
forcing extension $W[G]$ simultaneously extending all $W[c_i]$ for $i<n$. In particular, the sequence $\<c_0,c_1,\ldots,c_{n-1}>$ cannot be added by forcing over $W$.
\end{observation}

Let us turn now to infinite linearly ordered sequences of
forcing extensions. We show first in theorem
\ref{Theorem.IncreasingSequenceUnbounded} and observation
\ref{Observation.IncreasingSequenceUnbounded} that one
mustn't ask for too much; but nevertheless, in theorem
\ref{Theorem.IncreasingOmegaChainIsBounded} we prove the
surprising positive result, that any increasing sequence of
forcing extensions over a countable model $W$, with forcing
of uniformly bounded size, is bounded above by a single
forcing extension $W[G]$.

\begin{theorem}\label{Theorem.IncreasingSequenceUnbounded}
If $W$ is a countable model of \ZFC, then there is an
increasing sequence of set-forcing extensions of $W$ having
no upper bound in the generic multiverse of $W$.
$$W[G_0]\of W[G_1]\of\cdots\of W[G_n]\of\cdots$$
\end{theorem}

\begin{proof}  Since $W$ is countable, there is an increasing
sequence $\<\gamma_n\st n<\omega>$ of ordinals that is
cofinal in the ordinals of $W$. Let $G_n$ be $W$-generic
for the collapse forcing $\Coll(\omega,\gamma_n)$, as
defined in $W$. (By absorbing the smaller forcing, we may
arrange that $W[G_n]$ contains $G_m$ for $m<n$.) Since
every ordinal of $W$ is eventually collapsed, there can be
no set-forcing extension of $W$, and indeed, no model with
the same ordinals as $W$, that contains every
$W[G_n]$.\end{proof}

But that was cheating, of course, since the sequence of
forcing notions is not even definable in $W$, as the class
$\set{\gamma_n\st n<\omega}$ is not a class of $W$. A more
intriguing question would be whether this phenomenon can
occur with forcing notions that constitute a set in $W$, or
(equivalently, actually) whether it can occur using always
the same poset in $W$. For example, if $W[c_0]\of
W[c_0][c_1]\of W[c_0][c_1][c_2]\of\cdots$ is an increasing
sequence of generic extensions of $W$ by adding Cohen
reals, then does it follow that there is a set-forcing
extension $W[G]$ of $W$ with $W[c_0]\cdots[c_n]\of W[G]$
for every $n$? For this, we begin by showing that one
mustn't ask for too much:

\begin{observation}\label{Observation.IncreasingSequenceUnbounded}
If $W$ is a countable model of \ZFC, then there is a
sequence of forcing extensions $W\of W[c_0]\of
W[c_0][c_1]\of W[c_0][c_1][c_2]\of\cdots$, adding a Cohen
real at each step, such that there is no forcing extension
of $W$ containing the sequence $\<c_n\st n<\omega>$.
\end{observation}

\begin{proof}  Let $\<d_n\st n<\omega>$ be any $W$-generic
sequence for the forcing to add $\omega$ many Cohen reals
over $W$. Let $z$ be any real coding the ordinals of $W$.
Let us view these reals as infinite binary sequences.
Define the real $c_n$ to agree with $d_n$ on all digits
except the initial digit, and set $c_n(0)=z(n)$. That is,
we make a single-bit change to each $d_n$, so as to code
one additional bit of $z$. Since we have made only finitely
many changes to each $d_n$, it follows that $c_n$ is an
$W$-generic Cohen real, and also
$W[c_0]\cdots[c_n]=W[d_0]\cdots [d_n]$. Thus, we have $W\of
W[c_0]\of W[c_0][c_1]\of W[c_0][c_1][c_2]\of\cdots$, adding
a generic Cohen real at each step. But there can be no
forcing extension of $W$ containing $\<c_n\st n<\omega>$,
since any such extension would have the real $z$, revealing
all the ordinals of $W$ to be countable.\end{proof}

We can modify the construction to allow $z$ to be
$W$-generic, but collapsing some cardinals of $W$. For
example, for any cardinal $\delta$ of $W$, we could let $z$
be $W$-generic for the collapse of $\delta$. Then, if we
construct the sequence $\<c_n\st n<\omega>$ as above, but
inside $W[z]$, we get a sequence of Cohen real extensions
$W\of W[c_0]\of W[c_0][c_1]\of W[c_0][c_1][c_2]\of\cdots$
such that $W[\<c_n\st n<\omega>]=W[z]$, which collapses
$\delta$.

But of course, the question of whether the models
$W[c_0][c_1]\cdots[c_n]$ have an upper bound is not the
same question as whether one can add the sequence $\<c_n\st
n<\omega>$, since an upper bound may not have this
sequence. And in fact, this is exactly what occurs, and we
have a surprising positive result:

\begin{theorem} Suppose that $W$ is a countable model of \ZFC,
and
$$W[G_0]\of W[G_1]\of\cdots\of W[G_n]\of\cdots$$ is an increasing
sequence of forcing extensions of $W$, with $G_n\of\Q_n\in
W$ being $W$-generic. If the cardinalities of the $\Q_n$'s
in $W$ are bounded in $W$, then there is a set-forcing
extension $W[G]$ with $W[G_n]\of W[G]$ for all
$n<\omega$.\label{Theorem.IncreasingOmegaChainIsBounded}
\end{theorem}

\begin{proof}  Let us first make the argument in the special case
that we have
$$W\of W[g_0]\of W[g_0][g_1]\of\cdots\of
W[g_0][g_1]\cdots[g_n]\of\cdots,$$ where each $g_n$ is
generic over the prior model for forcing $\Q_n\in W$. That
is, each extension $W[g_0][g_1]\cdots[g_n]$ is obtained by
product forcing $\Q_0\cross\cdots\cross\Q_n$ over $W$, and
the $g_n$ are mutually $W$-generic. Let $\delta$ be a
regular cardinal with each $\Q_n$ having size at most
$\delta$, built with underlying set a subset of $\delta$.
In $W$, let $\theta=2^\delta$, let
$\<\R_\alpha\st\alpha<\theta>$ enumerate all posets of size
at most $\delta$, with unbounded repetition, and let
$\P=\prod_{\alpha<\theta}\R_\alpha$ be the finite-support
product of these posets. Since each factor is
$\delta^\plus$-c.c., it follows that the product is
$\delta^\plus$-c.c. Since $W$ is countable, we may build a
filter $H\of\P$ that is $W$-generic. In fact, we may find
such a filter $H\of\P$ that meets every dense set in
$\Union_{n<\omega}W[g_0][g_1]\cdots[g_n]$, since this union
is also countable. In particular, $H$ and
$g_0\cross\cdots\cross g_n$ are mutually $W$-generic for
every $n<\omega$. The filter $H$ is determined by the
filters $H_\alpha\of\R_\alpha$ that it adds at each
coordinate.

Next comes the key step. Externally to $W$, we may find an
increasing sequence $\<\theta_n\st n<\omega>$ of ordinals
cofinal in $\theta$, such that $\R_{\theta_n}=\Q_n$. This
is possible because the posets are repeated unboundedly,
and $\theta$ is countable in $V$. Let us modify the filter
$H$ by surgery to produce a new filter $H^*$, by changing
$H$ at the coordinates $\theta_n$ to use $g_n$ rather than
$H_{\theta_n}$. That is, let $H^*_{\theta_n}=g_n$ and
otherwise $H^*_\alpha=H_\alpha$, for
$\alpha\notin\set{\theta_n\st n<\omega}$. It is clear that
$H^*$ is still a filter on $\P$. We claim that $H^*$ is
$W$-generic. To see this, suppose that $A\of\P$ is any
maximal antichain in $W$. By the $\delta^+$-chain condition
and the fact that $\cof(\theta)^W>\delta$, it follows that
the conditions in $A$ have support bounded by some
$\gamma<\theta$. Since the $\theta_n$ are increasing and
cofinal in $\theta$, only finitely many of them lay below
$\gamma$, and we may suppose that there is some largest
$\theta_m$ below $\gamma$. Let $H^{**}$ be the filter
derived from $H$ by performing the surgical modifications
only on the coordinates $\theta_0,\ldots,\theta_m$. Thus,
$H^*$ and $H^{**}$ agree on all coordinates below $\gamma$.
By construction, we had ensured that $H$ and
$g_0\cross\cdots\cross g_m$ are mutually generic over $W$
for the forcing $\P\cross\Q_0\cross\cdots\Q_m$. This poset
has an automorphism swapping the latter copies of $\Q_i$
with their copy at $\theta_i$ in $\P$, and this
automorphism takes the $W$-generic filter $H\cross
g_0\cross\cdots\cross g_m$ exactly to $H^{**}\cross
H_{\theta_0}\cross\cdots \cross H_{\theta_m}$. In
particular, $H^{**}$ is $W$-generic for $\P$, and so
$H^{**}$ meets the maximal antichain $A$. Since $H^*$ and
$H^{**}$ agree at coordinates below $\gamma$, it follows
that $H^*$ also meets $A$. In summary, we have proved that
$H^*$ is $W$-generic for $\P$, and so $W[H^*]$ is a
set-forcing extension of $W$. By design, each $g_n$ appears
at coordinate $\theta_n$ in $H^*$, and so
$W[g_0]\cdots[g_n]\of W[H^*]$ for every $n<\omega$, as
desired.

Finally, we reduce the general case to this special case.
Suppose that $W[G_0]\of W[G_1]\of\cdots\of W[G_n]\of\cdots$
is an increasing sequence of forcing extensions of $W$,
with $G_n\of\Q_n\in W$ being $W$-generic and each $\Q_n$ of
size at most $\kappa$ in $W$. By the standard facts
surrounding finite iterated forcing, we may view each model
as a forcing extension of the previous model
$$W[G_{n+1}]=W[G_n][H_n],$$ where $H_n$ is
$W[G_n]$-generic for the corresponding quotient forcing
$\Q_n/G_n$ in $W[G_n]$. Let $g\of\Coll(\omega,\kappa)$ be
$\Union_n W[G_n]$-generic for the collapse of $\kappa$, so
that it is mutually generic with every $G_n$. Thus, we have
the increasing sequence of extensions $W[g][G_0]\of
W[g][G_1]\of\cdots$, where we have added $g$ to each model.
Since each $\Q_n$ is countable in $W[g]$, it is forcing
equivalent there to the forcing to add a Cohen real.
Furthermore, the quotient forcing $\Q_n/G_n$ is also
forcing equivalent in $W[g][G_n]$ to adding a Cohen real.
Thus, $W[g][G_{n+1}]=W[g][G_n][H_n]=W[g][G_n][h_n]$, for
some $W[g][G_n]$-generic Cohen real $h_n$. Unwrapping this
recursion, we have
$W[g][G_{n+1}]=W[g][G_0][h_1]\cdots[h_n]$, and consequently
$$W[g]\of W[g][G_0]\of W[g][G_0][h_1]\of W[g][G_0][h_1][h_2]\of\cdots,$$
which places us into the first case of the proof, since
this is now product forcing rather than iterated forcing.
\end{proof}

\begin{definition}\rm
A collection $\set{W[G_n]\st n<\omega}$ of forcing
extensions of $W$ is {\df finitely amalgamable} over $W$ if
for every $n<\omega$ there is a forcing extension $W[H]$
with $W[G_m]\of W[H]$ for all $m\leq n$. It is {\df
amalgamable} over $W$ if there is $W[H]$ such that
$W[G_n]\of W[H]$ for all $n<\omega$.
\end{definition}

The next corollary shows that we cannot improve the
non-amalgamability result of observation~\ref{Observation:NonAmalgamableExtensionsGeneralized} to the case of
infinitely many Cohen reals, with all finite subsets
amalgamable.

\begin{corollary}
If $W$ is a countable model of \ZFC\ and $\set{W[G_n]\st
n<\omega}$ is a finitely amalgamable collection of forcing
extensions of $W$, using forcing of bounded size in $W$,
then this collection is fully amalgamable. That is, there
is a forcing extension $W[H]$ with $W[G_n]\of W[H]$ for all
$n<\omega$.
\end{corollary}

\begin{proof}  Since the collection is finitely amalgamable, for
each $n<\omega$ there is some $W$-generic $K$ such that
$W[G_m]\of W[K]$ for all $m\leq n$. Thus, we may form the
minimal model $W[G_0][G_1]\cdots[G_n]$ between $W$ and
$W[K]$, and thus $W[G_0][G_1]\cdots [G_n]$ is a forcing
extension of $W$. We are thus in the situation of theorem
\ref{Theorem.IncreasingOmegaChainIsBounded}, with an
increasing chain of forcing extensions. $$W\of W[G_0]\of
W[G_0][G_1]\of\cdots\of W[G_0][G_1]\cdots[G_n]\of\cdots$$
Therefore, by theorem
\ref{Theorem.IncreasingOmegaChainIsBounded}, there is a
model $W[H]$ containing all these extensions, and in
particular, $W[G_n]\of W[H]$, as desired.\end{proof}

\section{The generic
mantle}\label{Section.TheGenericMantle}

Let us turn now to a somewhat enlarged context, where we
have found an intriguing and perhaps more robust version of
the mantle. The concept is inspired by
\cite{Fuchs2008:ClosedMaximalityPrinciples}, where the {\df
generic \HOD} is introduced, a topic we shall explore
further in section~\ref{section:TheGenericHOD}. The idea is
to consider not only the grounds of $V$, but also the
grounds of the forcing extensions of $V$. A {\df generic
ground} of $V$ is a ground of some set-forcing extension of
$V$, that is, a model $W$ having $V[G]=W[H]$ for some
$V$-generic filter $G\of\P\in V$ and $W$-generic filter
$H\of\Q\in W$. This is equivalent to saying that $V$ and
$W$ have a common forcing extension, and this is therefore
a symmetric relation. The {\df generic mantle} of $V$ is
the intersection of all generic grounds of $V$. Since this
includes $V$ itself, as well as every ground of $V$, it
follows that the generic mantle of $V$ will be included in
$V$ and indeed, in the mantle of $V$.

\begin{definition}\rm
The {\df generic mantle}, denoted $\gMantle$, is the class
of all $x$ such that every forcing notion $\P$ forces that
$\check{x}$ belongs to every ground of $V^\P$. That is,
$$ \gMantle=\set{x\st\forall\P\ \one\forces_\P\forall r\,\check{x}\in W_r},$$
where $W_r$ is the ground of the forcing extension defined
by parameter $r$ using the uniform definition of theorem
\ref{Theorem.ParameterizedGroundsW_r} applied in the
forcing extension. Equivalently,
$$ \gMantle=\set{x\st\forall\P\, \one\forces_\P\,\check{x}\in \Mantle},$$
where $\Mantle$ defines the mantle as interpreted in the
forcing extension.
\end{definition}

In particular, the generic mantle is uniformly definable
without parameters. If $\P$ is a poset, we shall write
$\Mantle^\P$ for the class
$\{x\st\one\forces_\P\check{x}\in\Mantle\}$, so that
$\gMantle=\bigcap_\P\Mantle^\P$.

\goodbreak
\begin{observation}\label{Observation.gMequalsIntersectionOfMantlesOfCollapses}%
$\gMantle=\bigcap_\alpha\Mantle^{\Coll(\omega,\alpha)}$.
\end{observation}

\begin{proof}  The inclusion from left to right is trivial, as the
intersection on the right ranges only over a subclass of
all forcing notions. For the reverse inclusion, suppose
that $x\in\bigcap_\alpha\Mantle^{\Coll(\omega,\alpha)}$,
and let $\P$ be an arbitrary poset, of some cardinality
$\alpha$. It follows that $\P\times\Coll(\omega,\alpha)$ is
forcing equivalent to $\Coll(\omega,\alpha)$, and so
$x\in\Mantle^{\Coll(\omega,\alpha)}=\Mantle^{\P\times\Coll(\omega,\alpha)}$.
But $\Mantle^{\P\times\Coll(\omega,\alpha)}\of\Mantle^\P$,
as is easily verified, so we are done.\end{proof}

In many models of set theory, the generic mantle and the
mantle coincide, and although we have introduced them as
distinct concepts, we do not actually know them to be
different, for we have no model yet in which we know them
to differ (see question
\ref{Question.MantleDifferentThanGenericMantle?}). We
introduce the generic mantle in part because the evidence
indicates that it may be a more robust notion. For example,
we are able to prove that the generic mantle satisfies the
\ZF\ axioms of set theory without needing the directedness
hypotheses of theorem~\ref{Theorem.IfDirectedMantleHasZF}.

\begin{theorem} The generic mantle of $V$ is a definable
transitive class in $V$, containing all ordinals, invariant
under set forcing, and a model of the \ZF\ axioms of set
theory.\label{Theorem.LimitMantleDefinableInvariant}
\end{theorem}

\begin{proof}
The generic mantle is a definable transitive class
containing all ordinals and closed under the \Godel\
operations, because it is the intersection of the generic
grounds, each of which has those closure properties in the
corresponding extension of $V$. So by fact~\ref{Fact.IM},
in order to show that $\gMantle$ is an inner model, it
remains to show merely that it is almost universal, which
we do below.

Let's first prove that $\gMantle$ is invariant under set
forcing, meaning that if $V[G]$ is a forcing extension of
$V$ by set forcing $\P$, then $\gMantle^V =
\gMantle^{V[G]}$. Observe first that every ground model of
a forcing extension of $V[G]$ is also a ground model of a
forcing extension of $V$. Thus, the generic mantle of
$V[G]$ is the intersection of a sub-collection of the
models used to form the generic mantle of $V$, and so
$\gMantle^V \subseteq \gMantle^{V[G]}$. For the reverse
inclusion, suppose that $x\notin \gMantle^V$. Thus, there
is a forcing extension $V[H]$ via some forcing notion $\Q$
for which there is a ground $W\of V[H]$ that does not
contain $x$. There is a condition $q\in H$ forcing that $W$
is like this. By forcing below $q$ with $\Q$ over $V[G]$,
we may assume that $H$ is not merely $V$-generic, but also
$V[G]$-generic, and consequently that $G\times H$ is
$V$-generic for the product forcing. Thus, $W$ is also a
ground of $V[H][G]$, which is the same as $V[G][H]$, and so
we have found a forcing extension of $V[G]$ that has a
ground $W$ that omits $x$. So $x\notin\gMantle^{V[G]}$ and
so $\gMantle^V=\gMantle^{V[G]}$, establishing that the
generic mantle is invariant by set forcing.

To see that the generic mantle is a model of \ZF, we claim
first that $V_\alpha\cap\gMantle\in\gMantle$, for every
$\alpha$. It is clear for any ground $W$ that
$V_\alpha\cap\gMantle=V_\alpha^W\cap\gMantle^W\in W$, by
the forcing invariance of $\gMantle$. Since this is true
for every ground, it follows that
$V_\alpha\cap\gMantle\in\Mantle$. And so we have shown in
\ZFC\ that for every $\alpha$, we have
$V_\alpha\cap\gMantle\in\Mantle$. So this is true in every
forcing extension, which again by the forcing absoluteness
of $\gMantle$ (and of $V_\alpha\cap\gMantle$) means that
$V_\alpha\cap\gMantle\in\Mantle^{\Coll(\omega,\gamma)}$,
for every $\gamma$ and every $\alpha$. So
$V_\alpha\cap\gMantle\in\gMantle$, for every $\alpha$, by
observation
\ref{Observation.gMequalsIntersectionOfMantlesOfCollapses}.
It now follows that if $A$ is a subset of $\gMantle$ of
rank $\alpha$, then $A\of V_\alpha\cap\gMantle\in\gMantle$,
and so $\gMantle$ is almost universal and hence is an inner
model, as desired.\end{proof}

The generic mantle of a model of set theory seems most
naturally considered in a broad context that includes all
of the forcing extensions of that model, their ground
models, the forcing extensions of those models, and so on.
Woodin~\cite{Woodin2004:CHMultiverseOmegaConjecture}
introduced the concept of the {\df generic multiverse} of a
model $U$ of set theory, the smallest collection of models
containing $U$ and closed under forcing extensions and
grounds, and this seems to be an ideal context in which to
undertake our project of set-theoretic geology. The generic
multiverse naturally partitions the larger multiverse of
models of set theory into equivalence (meta)classes,
consisting of models reachable from one another by passing
to forcing extensions and ground models.

Because the generic multiverse concept is clearly
second-order or higher-order, however, there are certain
difficulties of formalization and meta-mathematical issues
that need to be addressed. This is particularly true when
one wants to consider the generic multiverse of the full
set-theoretic universe $V$, rather than merely the generic
multiverse of a toy countable model. The standard
approaches to second-order set theory, after all, such as
\Godel-Bernays set theory or Kelly-Morse set theory, do not
seem to provide a direct account of the generic multiverse
of $V$, whose forcing extensions are of course not directly
available, even as \GBC\ or \KM\ classes.

Nevertheless, it turns out that many of the natural
questions one might want to ask about the generic
multiverse of $V$ can in fact be formalized in first-order
\ZFC\ set theory, and since we are indeed principally
interested in the features of the generic multiverse of the
full set-theoretic universe $V$, we shall prefer to
formalize our concepts this way whenever this is possible,
so that they would be available for that purpose. For
example, even though a naive treatment of the generic
mantle would seem to present difficulties arising from the
fact that we do not have full access inside $V$ to the
forcing extensions of $V$ and to their respective grounds,
we have nevertheless proved above that the generic mantle
is a definable class in $V$, and that it is invariant by
forcing over any model and to any ground. Thus, it is
entirely a first-order project to state that the generic
mantle satisfies a given statement or has such-and-such
relation to other definable classes, and such questions can
therefore be investigated without any meta-mathematical
difficulties.

The standard treatments of forcing over $V$, such as the
method of working under the forcing relation, the method of
Boolean-valued models or the naturalist account of forcing
(see~\cite{HamkinsSeabold:BooleanUltrapowers}), provide a
first-order means of treating truth in the forcing
extensions of $V$. Thus, it is a first-order statement of
set theory to state that a given assertion $\varphi$ holds
in some forcing extension of $V$ (and thus the modal
assertions in the modal logic of forcing as in
\cite{HamkinsLoewe2008:TheModalLogicOfForcing} are all
first-order expressible). Iterating this method, one can
similarly state in an entirely first-order manner whether
$\varphi$ holds in every forcing extension of every ground
of every forcing extension of $V$, or whether a given set
$x$ is an element of all such models, and so on. In this
way, we see that many natural questions about the nature of
the generic multiverse of $V$ are actually first-order
questions about $V$.

Nevertheless, it is sometimes illuminating to consider the
generic multiverse of a model in a context where we can
legitimately grasp the generic multiverse as a whole,
rather than only from the filtered perspective of one of
the models in it. In these circumstances, we temporarily
adopt what we call the {\df toy model} perspective. In this
method, which is analogous to the
countable-transitive-model approach to forcing as opposed
to the various forcing-over-$V$ approaches, one has a
countable model $W$ of \ZFC, and one considers the
collection of all forcing extensions of $W$, as constructed
in the current set-theoretic background universe $V$, as
well as the grounds of these universes and their forcing
extensions and so on, forming ultimately the smallest
collection of models closed under the forcing extensions
and grounds that are available. The particular generic
multiverse of $W$ therefore depends on the set-theoretic
background in which it is constructed; if we were to force
over $V$ to add a Cohen real $c$, for example, then for a
given countable transitive model $W$, the generic
multiverse of $W$ as computed in $V[c]$ will include
$W[c]$, but in $V$ this model is of course not available.
In any case, the generic multiverse of a countable model
$W$ will be a family of continuum many models, essentially
a set of reals under a suitable coding scheme. With the
generic multiverse thus becoming concrete for a given toy
model $W$, one may now imagine living inside $W$, treating
it as the full universe and using that constructed generic
multiverse as the generic multiverse of it.

It is essentially the toy model approach that one uses, for
example, when undertaking forcing only with countable
transitive models. Although this way of doing forcing was
formerly quite common, it has now largely been supplanted
with the various formalizations of forcing in \ZFC, which
enable one to make sense of forcing over an arbitrary model
of \ZFC. Since as we have mentioned we are principally
interested in the questions of set-theoretic geology and
the generic multiverse in the context of the full
set-theoretic universe $V$, we intend to fall back on the
toy approach only as a last resort, when it may not be
clear how a given concept might be formalized. For example,
we adopted the toy model approach in observation
\ref{Observation:NonAmalgamableExtensions}, since it is
unclear how to formalize that assertion as a claim about
the generic multiverse of the full universe $V$. Similarly,
the inner model hypothesis \IMH\ of
\cite{Friedman2006:InternalConsistencyAndIMH} is formalized
via the toy model approach, because it is unclear how to
formalize it as a claim about the full universe $V$.

The models of the generic multiverse can be specified by
the corresponding finite zigzag sequence of forcing
extensions and ground models on the path leading to it from
the original model; let us say that $\<U_0,U_1,\ldots,
U_k>$ is a {\df multiverse path} leading from $U_0$ to
$U_k$ if each $U_{i+1}$ is either a forcing extension or a
ground model of $U_i$. Woodin has argued that any statement
true in a model of the generic multiverse is true already
in a model reachable by a multiverse path of length at
most three from any fixed original model, specifically, in
a forcing extension of a ground model of a forcing
extension of the original model.\footnote{In
\cite{Woodin2004:CHMultiverseOmegaConjecture}, he appears
to make the stronger claim that the models themselves are
reachable in three steps, but in personal conversation with
the second author, he clarified this to the claim we have
just stated here. Meanwhile, under the hypothesis of
theorem~\ref{Theorem.TFAEgenericDDG}, every model in the
generic multiverse of $V$ is reachable in only two steps,
as a ground extension of $V$.} This three-steps-away fact
(or indeed, any finite bound on the number of steps) allows
us to express in a first-order manner the
generic-multiverse ``possibility'' modality
$\possible\varphi$, asserting that $\varphi$ is
generic-multiverse possible, that is, that $\varphi$ is
true in some model of the generic multiverse. This is
because by the three-step fact, $\possible\varphi$ is
equivalent to the assertion that there is a forcing notion
$\P$ forcing that there is a ground having a forcing notion
forcing $\varphi$ over that ground. These forcing
modalities are further explored in
\cite{HamkinsLoewe2008:TheModalLogicOfForcing},
\cite{HamkinsLeibmanLoewe:StructuralConnectionsForcingClassAndItsModalLogic},
\cite{HamkinsLoewe2013:MovingUpAndDownInTheGenericMultiverse}, with the
latter article specifically treating the interaction of the
upward and downward forcing possibility modalities and
offering improvements and further results on this
particular issue.

Woodin introduced the generic multiverse in
\cite{Woodin2004:CHMultiverseOmegaConjecture} in part to
make a philosophical argument against a certain view of
mathematical truth, namely, truth as
true-in-the-generic-multiverse. Although we investigate the
generic multiverse, we do not hold this view of truth.
Rather, we are attempting to understand the fundamental
features of the generic multiverse, a task we place at the
foundation of any deep understanding of forcing, and we
consider the generic multiverse to be the natural and
illuminating background context for our project of set-theoretic geology. Apart from any grand
Platonist-multiverse view of truth, we find that the
ordinary Tarskian view of truth, considered individually
within each model of the multiverse, is sufficient to
provide a global understanding of truth across the
multiverse, for the definability of the forcing relation
gives every model access to the truth concepts of its
forcing extensions, and conversely, the ground-model definability theorem
gives every forcing
extension direct access to its grounds. Our view is that
the generic multiverse is a grand new set-theoretic context
for us to explore: what are the features of the generic
multiverse? What is or can it be like? This is what
set-theoretic geology is about.\footnote{Haim Gaifman has
pointed out (with humor) that the term `set-theoretic
geology' suggests a Platonist view of forcing as a natural
process or force of nature, and that for those who view
forcing as an intentional human activity, the alternative
term should be `set-theoretic archeology.'}

\begin{corollary}
The generic mantle is constant across the generic
multiverse. Indeed, the generic mantle is the intersection
of the generic multiverse.
\label{Corollary.LimitMantleIsIntersectionMultiverse}
\end{corollary}

\begin{proof}
This is an immediate consequence of theorem
\ref{Theorem.LimitMantleDefinableInvariant}. Specifically,
since the generic mantle is invariant by set forcing, it is
preserved as one moves either from a ground to a forcing
extension or from a forcing extension to a ground. Since
these are the operations that generate the multiverse, all
the models in the multiverse have the same generic mantle.
This means, in particular, that the generic mantle is
contained within the intersection of the multiverse.
Conversely, the generic mantle is defined to be the
intersection of some of the models of the multiverse, and
so the intersection of the multiverse is contained within
the generic mantle. Thus, they are equal.\end{proof}

\begin{corollary}\label{Corollary.gMantleIsLargestForcingInvariantClass}
The generic mantle is the largest forcing-invariant
definable class.\end{corollary}

\begin{proof}
We have argued above that the generic mantle is a
forcing-invariant definable class in any model of set
theory. Any other class that is definable and invariant by
forcing over any model of set theory will be preserved to
any forcing extension $V[G]$ and then to any subsequent
ground $W$ of $V[G]$, and therefore be contained within the
generic mantle. So the generic mantle is the largest such
forcing-invariant class.
\end{proof}

To our way of thinking, the previous corollary identifies
the generic mantle as a highly canonical object, the
largest forcing-invariant definable class. Because of this,
we would expect it to become a focus of attention for those
set theorists interested in forcing and models of set
theory.

We would like now to consider the downward-directed grounds
hypothesis as it arises in connection not only with the
grounds of $V$, but with the grounds of the forcing
extensions of $V$. Recall that the {\df generic grounds} of
a model are the grounds of its forcing extensions. We say
that the grounds are {\df dense below the generic grounds,}
if every generic ground contains a ground. This concept is
first-order expressible in set theory as the assertion that
for every poset $\P$ and every $\P$-name $\dot{r}$, every
condition in $\P$ forces ``there is an index
$s\in\check{V}$ such that $(W_s)^{\check{V}}\of W_{\dot
r}$''. A {\df ground extension} of $V$ is a forcing
extension of a ground of $V$, or in other words, a model
having a common ground with $V$; all such models have the
form $W_r[G]$ for some $W_r$-generic filter $G\of\P\in
W_r$, but we may not assume here that $G$ is $V$-generic.

\begin{lemma}
Any generic ground of $V$ that is contained in $V$ is a
ground of $V$. \label{Lemma.GenericGroundsInV}
\end{lemma}

\begin{proof}
If $W$ is a ground of a forcing extension $V[G]$ and $W\of
V\of V[G]$, then by lemma
\ref{Fact.IntermediateModelsAreGrounds} it follows that $W$
is a ground of $V$.
\end{proof}

What we don't know,  more generally, is the following
analogue of lemma~\ref{Lemma.GenericGroundsInV} for models
in the generic multiverse.

\begin{question}
If $W$ is in the generic multiverse of $V$ and $W\of V$,
must $W$ be a ground of $V$? In other words, for models
within the same generic multiverse, is the inclusion
relation the same as the ``is a ground model of'' relation?
\end{question}

The answer is yes if the grounds are dense in the
multiverse, since then every model in the multiverse
contained in V is trapped between a ground of $V$ and $V$,
and hence is itself a ground by lemma
\ref{Lemma.GenericGroundsInV}.

\begin{theorem} The following are equivalent.
 \begin{enumerate}
  \item The \DDG\ holds. That is, the grounds of $V$
      are downward directed.
  \item The \DDG\ holds in some forcing extension of
      $V$.
  \item The \DDG\ holds in every ground of $V$.
 \end{enumerate}
\label{Theorem.DownwardDirectedIffSomeExtension}
\end{theorem}

\begin{proof}
Since $V$ is a forcing extension of itself, the
implications $1{\implies} 2$ and $3{\implies}1$ are
immediate. For $2{\implies}3$, suppose that the grounds of
a forcing extension $V[G]$ are downward directed and that
$W$ and $W'$ are grounds of $U$, which is a ground of $V$.
Thus, $W$ and $W'$ are also grounds of $V[G]$, and so by 2
there is some ground $\bar{W}$ of $V[G]$ with $\bar{W}\of
W\intersect W'$. Since $\bar{W}\of U\of V[G]$, it follows
by lemma~\ref{Lemma.GenericGroundsInV} that $\bar{W}$ is a
ground of $U$, and so $3$ holds.
\end{proof}

Consider now in contrast the {\it generic} \DDG, which
asserts that the \DDG\ holds in all forcing extensions.

\begin{theorem}\label{Theorem.TFAEgenericDDG}
The following are equivalent.
\begin{enumerate}
 \item The generic \DDG\ holds. That is, in every
     forcing extension of\/ $V$, the grounds are
     downward directed.
 \item The grounds of\/ $V$ are downward directed and
     dense below the generic grounds.
 \item The grounds of\/ $V$ are downward directed and
     dense below the grounds of every ground extension.
\end{enumerate}
\end{theorem}

\begin{proof}
$(1\implies 3)$ Suppose that the grounds of every forcing
extension of $V$ are downward directed. In particular,
since $V$ is a forcing extension of itself, the grounds of
$V$ are downward directed. Suppose that $W$ is a ground of
a ground extension $W_r[G]$, where $G\of\P\in W_r$ is
$W_r$-generic. Thus, $W$ has the form $W=W_t^{W_r[G]}$ for
some parameter $t\in W_r[G]$, having a suitable $\P$-name
with $t={\dot t}_G$.

If we temporarily imagine that $G$ is actually $V$-generic
for the forcing $\P$, then we may also form $V[G]$ and
notice that $W_r[G]$ is a ground of $V[G]$. In this case,
$V$, $W$ and $W_r$ are all grounds of $V[G]$, and so by our
assumption there is a ground $W_s^{V[G]}$ contained in each
of them. By lemma~\ref{Lemma.GenericGroundsInV}, it follows
that $W_s^{V[G]}$ is a ground of $V$, which we may
therefore denote simply by $W_s$, as well as a ground of
$W_r$. Thus, in this case where $G$ is $V$-generic, we have
found the desired ground $W_s$ of $V$ contained in $W$.
Since $W_r[G]$ can see that $W_s\of W_{\dot t_G}^{W_r[G]}$,
there must be a condition $p\in \P$ forcing this property
about $s$ over $W_r$. Furthermore, by arguing the same
below any given condition, $W_r$ can see that the set of
such $p$ having such an $s$ is dense in $\P$.

If we now drop the assumption that $G$ is $V$-generic, we
may nevertheless assume by genericity that there is a
condition $p\in G$ with $s$ as above, ensuring that
$W_s^{W_r}\of W$. But $W_s^{W_r}$ is a ground of $W_r$ and
therefore a ground of $V$. So we have found a ground of $V$
in $W$, as desired for 3.

$(3\implies 2)$ This is immediate, since any ground of a
forcing extension of $V$ is a ground of a ground extension,
since $V$ is a ground of itself.

$(2\implies 1)$ Suppose that the grounds of $V$ are
downward directed and dense below the generic grounds, and
suppose that $V[G]$ is a forcing extension of $V$, with
grounds $W_r^{V[G]}$ and $W_s^{V[G]}$. Since the grounds
are dense below the generic grounds, we may find grounds
$W_{r'}$ and $W_{s'}$ with $W_{r'}\of W_r^{V[G]}$ and
$W_{s'}\of W_s^{V[G]}$. Since the grounds are downward
directed, there is a ground $W_t$ with $W_t\of
W_{r'}\intersect W_{s'}$, and this is contained in
$W_r^{V[G]}\intersect W_s^{V[G]}$. Since $W_t$ is also a
ground of $V[G]$, we conclude that the grounds of $V[G]$
are downward directed, as desired.
\end{proof}

\begin{corollary}
If the generic \DDG\ holds, then
 \begin{enumerate}
  \item the mantle is the same as the generic mantle;
  \item the class of ground extensions is closed under
      forcing extensions and grounds;
  \item and consequently, the generic multiverse of $V$
      consists of the ground extensions of $V$.
 \end{enumerate}
\end{corollary}

\begin{proof}
Suppose that the generic \DDG\ holds. By theorem
\ref{Theorem.TFAEgenericDDG}, the grounds are dense below
the generic grounds, and so the mantle and generic mantle
coincide, establishing statement (1). By the same theorem,
any ground $W$ of a ground extension $W_r[G]$ contains some
ground $W_s$ of both $V$ and $W_r$, and so $W_s\of W\of
W_r[G]$. Since $W_s$ is a ground of $W_r[G]$, it follows by
lemma~\ref{Lemma.GenericGroundsInV} that $W$ is a forcing
extension of $W_s$, and so $W$ is a ground extension of
$V$. Thus, the collection of ground extensions of $V$ is
closed under grounds and also (clearly) under forcing
extensions, establishing statement (2). Statement (3) now
follows easily.
\end{proof}

\begin{theorem}
If there is a model of \ZFC, then there is a model of \ZFC\
having a ground extension that is not a generic ground. In
particular, the generic grounds of such a model do not
exhaust its generic multiverse.
\end{theorem}

\begin{proof}
Suppose that $W$ is a countable model of \ZFC\ and consider
the non-amal-gamable extensions $W[c]$ and $W[d]$ arising
in observation~\ref{Observation:NonAmalgamableExtensions}.
Focus on the model $W[c]$, and observe that the model
$W[d]$ is a forcing extension of a ground of $W[c]$, but
there can be no extension $W[c][G]$ of which $W[d]$ is a
ground, since any such extension would amalgamate $W[c]$
and $W[d]$.
\end{proof}

On the positive side, it is also relatively consistent with
\ZFC\ that the generic grounds do exhaust the generic
multiverse. This is true in $L$, for example, since the
generic multiverse of $L$ consists precisely of the forcing
extensions of $L$. We say that a generic ground $W$ is a
{\df generic bedrock} if it is minimal among the generic
grounds (meaning that there is no generic ground which is
properly contained in it) and it is a {\df solid generic
bedrock} if it is least among the generic grounds (meaning
that it is contained in every generic ground). Since a
generic bedrock is contained in the generic ground $V$, it
follows by lemma~\ref{Lemma.GenericGroundsInV} that it is a
ground, and it is straightforward to formalize the
statement that a ground $W_s$ is a bedrock or a solid
bedrock. Let the {\df solid generic bedrock axiom} assert
that there is a solid generic bedrock. By lemma
\ref{Lemma.GenericGroundsInV} again, the solid generic
bedrock axiom is equivalent to the assertion that the
generic mantle is a ground (and one may take this as the
formal definition of this axiom). In the presence of the
solid generic bedrock axiom, of course, the generic grounds
will be downward set-directed, and furthermore, the generic
multiverse will consist precisely of the forcing extensions
of the solid generic bedrock.

Let us consider now the set-directed analogue of theorem
\ref{Theorem.TFAEgenericDDG}.

\goodbreak
\begin{theorem} The following are equivalent
\begin{enumerate}
 \item The generic strong \DDG\ holds. That is, in
     every forcing extension of $V$, the grounds are
     downward set-directed.
 \item The grounds of $V$ are downward set-directed and
     dense below the generic grounds.
 \item The grounds of $V$ are downward set-directed and
     dense below the grounds of every ground extension.
\end{enumerate}
\end{theorem}

\begin{proof}
$(1\implies 3)$ If in every forcing extension, the grounds
are downward set-directed, then in particular in $V$ they
are downward set directed. And they are dense below the
grounds of any ground extension by theorem
\ref{Theorem.TFAEgenericDDG}, since the generic strong
\DDG\ implies the generic \DDG, so statement $3$ holds.

$(3\implies 2)$ This is immediate, since every generic
ground is trivially a ground of a ground extension, since
$V$ is a ground of $V$.

$(2\implies 1)$ Suppose that the grounds are downward
set-directed and every generic ground contains a ground.
Consider any forcing extension $V[G]$, where $G\of\P$ is
$V$-generic, and any set $I$ of indices in $V[G]$, giving
rise to the generic grounds $W_r^{V[G]}$ for $r\in I$. By our
assumption, for every such $r\in I$ there is $t\in V$ with
$W_t\of W_r^{V[G]}$. If $\dot r$ is a name for $r$, then
there is a condition $p\in G$ forcing that $W_{\check
t}^{\check V}\of W_{\dot r}^{\check V[\dot G]}$. Fix a name
$\dot I$ for $I$ that is full, in the sense that whenever
$\boolval{\dot r_0\in\dot I}\neq 0$, then there is $\dot
r\in\dom(\dot I)$ such that $\boolval{\dot r_0=\dot r}=1$.
By collecting the witnesses $t$ as above for each $\dot
r\in\dom(\dot I)$ and condition $p\in\P$, we may find a set
$J\in V$ such that whenever $p\forces \dot r\in \dot I$,
then there is $q\leq p$ and $t\in J$ such that $q\forces
W_{\check t}^{\check V}\of W_{\dot r}^{\check V[\dot G]}$.
A simple density argument now shows that for every $r\in
I$, there is $t\in J$ with $W_t\of W_r^{V[G]}$. Next, by
the downward directedness of the grounds in $V$, there is a
single ground $W_s$ contained in $W_t$ for all $t\in J$ and
hence also in $W_r^{V[G]}$ for every $r\in I$. Since $W_t$
is a ground of $V$ and hence also of $V[G]$, we have
therefore established that the grounds of $V[G]$ are
downward directed, as desired.
\end{proof}

Consider now the local version of downward directedness.
Extending our previous terminology, let us define the {\df
local} \DDG\ hypothesis as the assertion that the grounds
are locally downward directed, meaning that for every $r$
and $s$ and every set $B$ there is $t$ with $W_t\intersect
B\of W_r\intersect W_s$; the {\df strong local} \DDG\
asserts that the grounds are locally downward set-directed,
meaning that for every set $I$ and set $B$ there is $t$
with $W_t\intersect B\of \Intersect_{r\in I}W_r$; and the
{\df generic strong local} \DDG\ asserts that this holds in
every forcing extension.

\begin{theorem}\label{Theorem.LocallyRealizedLimitMantleHasZFC}
If the generic strong local \DDG\ holds, then the generic
mantle is a model of \ZFC.
\end{theorem}

\begin{proof}
We already know by theorem
\ref{Theorem.LimitMantleDefinableInvariant} that the
generic mantle $\gMantle$ satisfies \ZF. Unfortunately, we
cannot apply corollary~\ref{Corollary.IntersectionGivesZF}
directly here in order to prove the axiom of choice holds
in the generic mantle, because the generic grounds are not
classes in $V$. But we can adapt the argument of theorem
\ref{Theorem.IfDirectedMantleHasZF}. Assume toward
contradiction that $y$ is a member of $\gMantle$ which has
no well-order in $\gMantle$. Let $R$ be the set of
well-orders of $y$. Since none of these are in $\gMantle$,
we may find for each $r\in R$ a poset $\P_r$ which forces
``$\check{r}\notin \Mantle^{V[\dot G_r]}$,'' where $\dot
G_r$ is the canonical $\P_r$-name of the generic filter.
Let $\P=\prod_r\P_r$ be the full-support product of these
forcing notions and suppose $G\of\P$ is $V$-generic. In
$V[G]$, we have the extension $V[G_r]$ arising from
coordinate $r$, and by the choice of $\P_r$ it follows that
$r$ is not in the mantle of $V[G_r]$, and hence also not in
the mantle of $V[G]$. Thus, in $V[G]$, for each $r\in R$ we
may find a ground $W_{t_r}^{V[G]}$ omitting $r$. Since we
have assumed that the grounds of $V[G]$ are locally
downward set directed, there is a single ground $W\of V[G]$
omitting every $r\in R$. Since $y\in\gMantle^V$, it follows
that $y\in W$, but we have proved that $y$ has no
well-ordering in $W$, contrary to $W\satisfies\ZFC$.
\end{proof}

The generic strong local \DDG\ follows of course from the
generic strong \DDG, which we have mentioned holds
trivially in $L$, for example, simply because $L$ is a
ground below any ground of any forcing extension of $L$. In
other words, $L$ is a solid generic bedrock. There are many
other models having this feature, but we do not fully know
the extent of the phenomenon.

\begin{question}\label{Question.LimitManleIsGround?}
When does the universe have a solid generic bedrock? In
other words, under what circumstances is the generic mantle
also a ground model of the universe?
\end{question}

\noindent This phenomenon is not universal, in light of the
following.

\begin{corollary}\label{Corollary.Bottomless}%
Every model $V$ of \ZFC\ has a class forcing extension
$V[G]$ that satisfies the generic strong local \DDG, but
which has no bedrock and no generic bedrock.
\end{corollary}

\begin{proof}
This is a corollary to theorem
\ref{Theorem.DirectedButNoBedrock}. We first move to
$\Vbar$ satisfying the \CCA. Then, we perform an Easton
support product $\P$ adding a Cohen subset over $\Vbar$ to
every regular cardinal $\lambda$ for which
$2^\ltlambda=\lambda$. Suppose that $\Vbar[G]$ is the
resulting extension, and that $\Vbar[G][g]$ is obtained by
further set forcing over $\Vbar[G]$. If $W$ is a ground of
$\Vbar[G][g]$, then since the sets in $\Vbar$ remain coded
unboundedly in the \GCH\ pattern of $\Vbar[G][g]$, it
follows that $\Vbar\of W$. Since $\Vbar[G][g]$ is obtained
by set forcing over $W$, it follows that for $\lambda$
above the size of this forcing, since the generic sets
added by $\P$ at stage $\lambda$ have all their initial
segments in $\Vbar$ and hence in $W$, that these generic
objects must be already in $W$. It follows as in theorem~\ref{Theorem.DirectedButNoBedrock} that
$\Vbar[G^\alpha]\of W\of\Vbar[G][g]$ for some sufficiently
large $\alpha$. Thus, once again the models
$\Vbar[G^\alpha]$ form a strictly descending sequence of
grounds, which are dense below the generic grounds of
$\Vbar[G]$. It follows that the generic grounds are
downward set-directed, but there is no minimal ground or
generic ground, as desired.\end{proof}

As before, this can be strengthened by the following
corollary to theorem
\ref{Theorem.ConstructibleFromSetButNonBA}. Again, the
model obtained there does not have the property that every
set in it is set-generic over $V$, in contradistinction to
the one constructed above.

\begin{corollary}\label{Corollary.NoGenericBedrockWhenUniverseIsConstructibleFromASet}
Every model of set theory has a class forcing extension of
the form $L[r]$, where $r\sub\omega$, which satisfies the
generic strong local \DDG, but which has no bedrock or
generic bedrock.
\end{corollary}

\begin{proof}
This is a consequence of theorem
\ref{Theorem.ConstructibleFromSetButNonBA}. Starting with
an arbitrary model of set theory, by that theorem, there is
a proper class forcing extension of the form $L[r]$ in
which the bedrock axiom fails. Since this model is
constructible from a set, we know by theorem
\ref{Theorem.IfConstructibleFromASetThenSetDirectedGrounds}
that the generic grounds are downward set-directed.

Now the original argument showing that $L[r]$ has no
bedrock also shows that it has no generic bedrock. The
point is that any forcing extension $L[r][g]$ of $L[r]$ has
class many grounds, just like $L[r]$, as any ground of
$L[r]$ is a ground of $L[r][g]$. But $L[r][g]$ is still
constructible from a set, so by theorem
\ref{Theorem:IfConstructibleFromASetThenSBAEquivToSetManyGrounds},
it follows that the bedrock axiom fails in $L[r][g]$. So
there is no generic bedrock, since a generic bedrock would
be a bedrock in some set-forcing extension.\end{proof}

The argument of the previous proof can be expanded to show
the following.

\begin{theorem}
If the universe is constructible from a set, then this is
true throughout the generic multiverse. If in addition the
bedrock axiom fails, then it fails throughout the generic
multiverse.
\end{theorem}

\begin{proof}
We have to show that the properties in question are
preserved when passing to set-forcing extensions and to
grounds. If the universe is constructible from a set, then
this is obviously true in forcing extensions. To see that
it is true in grounds as well, suppose that $V=L[x]$ and
$W$ is a ground of $V$. Let $\P\in W$ be a partial order
for which $G$ is generic, such that $L[x]=W[G]$. Let $\tau$
be a name such that $x=\tau_G$. Let $a\in W$ be a set of
ordinals which codes the transitive closure
$t=\mathsf{TC}(\{\P,\tau\})$, and consider the \ZFC\ model
$L[a]$. Notice that $L[a]\of W$ since $a\in W$ and
$L[a][G]=L[x]$, since $x=\tau_G\in L[a][G]$. Moreover, $G$
is $L[a]$-generic, and so $L[a]\of W\of L[a][G]$ traps the
\ZFC\ model $W$ between $L[a]$ and its forcing extension
$L[a][G]=L[x]$, and so it follows by fact
\ref{Fact.IntermediateModelsAreGrounds} that $W$ is a
set-forcing extension of $L[a]$, and hence constructible
from a set.

Let's now assume that the universe is constructible from a
set and the bedrock axiom fails. Clearly, the failure of
the bedrock axiom is preserved when passing to grounds,
since a bedrock of a ground would be a bedrock, by fact
\ref{Fact.IntermediateModelsAreGrounds}. This holds in
general, without assuming that the universe is
constructible from a set. The argument establishing
corollary
\ref{Corollary.NoGenericBedrockWhenUniverseIsConstructibleFromASet}
shows that the failure of the bedrock axiom persists from
models which are constructible from a set to their
set-forcing extensions.\end{proof}

Let us say that a model having no bedrock model is {\df
bottomless} and a model having no generic bedrock is {\df
generically bottomless}.

We are of course interested more generally in the question
of in what sense the generic multiverse itself may be
downward directed, downward set-directed or locally
downward set-directed. These questions are easy enough to
formalize directly in the toy model approach, but in the
full context of forcing over $V$, the questions can present
meta-mathematical difficulties for the general case,
although they are also consequences of other axioms, such
as the solid generic bedrock axiom, which are formalizable.
At least part of the interest in such questions, of course,
has to do with their relation to the following question.

\begin{question}
Does the generic mantle necessarily satisfy \ZFC?
\end{question}

Although we have introduced the mantle and the generic
mantle as distinct notions, as we mentioned earlier we do not actually know that
they are different. We have as yet no models in which the
mantle differs from the generic mantle.

\begin{question}
Is it consistent with \ZFC\ that the mantle is different
from the generic
mantle?\label{Question.MantleDifferentThanGenericMantle?}
\end{question}

At the end of the next section, we shall see that all the
previous questions have simple answers if the universe is
constructible from a set.

\section{The generic \HOD}%
\label{section:TheGenericHOD}

Let us now introduce the generic \HOD, a concept
generalizing the classical \HOD\ in the same way that the
generic mantle generalizes the mantle. We assume that the
reader is familiar with the basic theory of $\HOD$, the
class of hereditarily ordinal definable sets (consult
\cite{Jech:SetTheory3rdEdition} for a review).

\begin{definition}\rm
The {\df generic \HOD} of $V$, denoted $\gHOD$, is the
intersection of all the {\HOD}s of all set-forcing
extensions of $V$. That is, $x\in\gHOD$ if and only if
$x\in\HOD^{V[G]}$ for all
set-forcing extensions $V[G]$. So%
$$\gHOD=\{x\st\forall\P\ \one\forces_\P\check{x}\in\HOD\}.$$%
\label{Definition.gHOD}
\end{definition}

The generic \HOD\ has been referred to as the limit \HOD,
or $\lim_\omega\HOD$ when it was introduced in
\cite[p.~298]{Fuchs2008:ClosedMaximalityPrinciples} as a
special case of $\lim_\kappa\HOD$. The following
characterization was used there.

\begin{lemma}
$\gHOD=\{x\st\forall\alpha\,\forces_{\Coll(\omega,\alpha)}\check{x}\in\HOD\}=\bigcap_{\alpha<\Ord}\HOD^{V^{\Coll(\omega,\alpha)}}$.
\label{Lemma.gHODIsIntersectionOfHODsOfCollapses}
\end{lemma}

\begin{proof}  Only the first identity needs a proof. The
inclusion from left to right is trivial. For the opposite
direction, let $x$ be in $\HOD^{V^{\Coll(\omega,\alpha)}}$,
for every $\alpha$. Let $\P$ be an arbitrary poset, and let
$\alpha$ be at least the cardinality of $\P$. Then
$\P\times\Coll(\omega,\alpha)$ embeds densely into
$\Coll(\omega,\alpha)$. Let $G\times H$ be
$\P\times\Coll(\omega,\alpha)$-generic over $V$, and let
$H'$ be $\Coll(\omega,\alpha)$-generic over $V$ such that
$V[G][H]=V[H']$. Then it follows from the homogeneity of
$\Coll(\omega,\alpha)$, and from the fact that
$\Coll(\omega,\alpha)$ is ordinal definable, that
$$x\in\HOD^{V[H']}=\HOD^{V[G][H]}\of\HOD^{V[G]}.$$
So $\P$ forces that $x\in\HOD$, and as $\P$ was arbitrary,
it follows that $x\in\gHOD$, as desired.
\end{proof}

The proof of the previous lemma also shows that if $\alpha\le\beta$, then
$\HOD^{V^{\Coll(\omega,\beta)}}\of\HOD^{V^{\Coll(\omega,\alpha)}}$, as we had noted in theorem \ref{Theorem.IfConstructibleFromASetThenSetDirectedGrounds}.

\begin{theorem}[\cite{Fuchs2008:ClosedMaximalityPrinciples}] In any model of set
theory, $\gHOD$ is a parameter-free uniformly first-order
definable class, containing all ordinals, invariant by set
forcing, and a model of
\ZFC.\label{Theorem.gHODisInvariant}
\end{theorem}

\begin{proof}
The standard treatments of $\HOD$ show that the relation
$x\in\HOD$ is parameter-free uniformly first-order
definable in any model of set theory. So definition
\ref{Definition.gHOD} is a definition of $\gHOD$ which has
the desired form. It is easy to see that $\gHOD$ is
transitive and that it contains all ordinals.

To see that it is invariant by forcing, consider any
set-forcing extension $V\of V[G]$ by a poset $\P$. If
$\alpha$ is at least the cardinality of $\P$ and $H$ is
$\Coll(\omega,\alpha)$-generic over $V[G]$, then there is
an $H'$ which is $\Coll(\omega,\alpha)$-generic over $V$
such that $V[H']=V[G][H]$. So
$(\HOD^{V^{\Coll(\omega,\alpha)}})^{V[G]}=\HOD^{V[G][H]}=\HOD^{V[H']}=\HOD^{V^{\Coll(\omega,\alpha)}}$.
So by the remark after lemma~\ref{Lemma.gHODIsIntersectionOfHODsOfCollapses}, it follows that
$$\gHOD^{V[G]}=\bigl(\!\bigcap_{|\P|\le\beta}\HOD^{V^{\Coll(\omega,\beta)}}\bigr)^{V[G]}
=\!\!\bigcap_{|\P|\le\beta}\!\!\HOD^{V^{\Coll(\omega,\beta)}}=\gHOD.$$

It follows immediately from corollary
\ref{Corollary.IntersectionGivesZF} that
$\gHOD\models\ZFC$, as it is the intersection of the
definable decreasing sequence
$\<\HOD^{V^{\Coll(\omega,\alpha)}}\st\alpha<\Ord>$ of
\ZFC-models. \end{proof}

\begin{corollary} The generic \HOD\ is constant across the
generic multiverse and is contained in the generic mantle.
The classes exhibit the following inclusions:
$$\begin{diagram}[height=2em]
\HOD & \\
\text{\rotatebox{90}{$\subseteq$}} \\   
\gHOD & \of & \gMantle & \of & M \\
\end{diagram}$$\label{Corollary:RelationshipBetweenHOD,gHOD,gMandM}
\end{corollary}

\begin{proof}  Because theorem~\ref{Theorem.gHODisInvariant} shows
that $\gHOD$ is invariant by forcing, every model in the
multiverse has the same $\gHOD$. Thus, $\gHOD$ is contained
within the intersection of the multiverse, which is the
generic mantle, and so $\gHOD\of\gMantle$. The other
inclusions $\gHOD\of\HOD$ and $\gMantle\of \Mantle$ are
immediate.\end{proof}

These basic inclusion relations will be separated by the
theorems of the next section, except that the exact nature
of the generic mantle remains somewhat unsettled, since we
have been unable to separate it either from the generic
\HOD\ or from the mantle, although we do separate these
latter two. Before moving on to those results, however, let
us first prove the following theorem, which explains in
part why proper class forcing will loom so large in our
subsequent arguments.

\begin{theorem} If the universe is constructible from a set, $V=L[a]$, then%
$$ \gHOD=\gMantle=\Mantle.$$
\end{theorem}

\begin{proof}  We know already that $\gHOD\of\gMantle\of\Mantle$,
by corollary
\ref{Corollary:RelationshipBetweenHOD,gHOD,gMandM}, so the
only thing left to prove is that $\Mantle\of\gHOD$. But
this follows immediately from theorem
\ref{Theorem.IfConstructibleFromASetThenSetDirectedGrounds},
which states among other things that the inner models of
the form $\HOD^{\V^{\Coll(\omega,\alpha)}}$ are dense in
the grounds, together with lemma
\ref{Lemma.gHODIsIntersectionOfHODsOfCollapses}, which says
that the generic \HOD\ is the intersection of these models.
\end{proof}

\section{Controlling the mantle and the generic mantle}

We now prove our main theorems, which control the mantle
and generic mantle of the target models, and also the \HOD\
and generic \HOD.

\begin{theorem} Every model $V$ of \ZFC\ has a class forcing
extension $V[G]$ in which $V$ is the mantle, the generic
mantle, the generic \HOD\ and the \HOD.
\label{Theorem.V=M=gM=gHOD=HOD}%
$$V=\Mantle^{V[G]}=\gMantle^{V[G]}=\gHOD^{V[G]}=\HOD^{V[G]}$$
\end{theorem}

\begin{proof}  Our strategy will be to perform class forcing $V\of
V[G]$ in such a way that the various forces acting on the
mantles and {\HOD}s in $V[G]$ are perfectly balanced, in
each case giving $V$ as the result. Pushing upward,
expanding these classes up to $V$, we will force in such a
way that every set in $V$ is coded explicitly into the
continuum function of $V[G]$, thereby ensuring that each
such set is in the mantle, the generic mantle, the \HOD\
and the generic \HOD. Pressing downward, holding these
classes down to $V$, we will maintain certain factor and
homogeneity properties on the forcing that ensure that no
additional sets are added to the mantles and {\HOD}s.

For each ordinal $\alpha$, let $\delta_\alpha$ be the
$\alpha^{\rm th}$ cardinal of the form $\lambda^\plus$, where
$\lambda$ is a strong limit cardinal, but not a limit of
strong limit cardinals. More precisely,
$\delta_\alpha=\beth_{\omega\cdot(\alpha+1)}^\plus$, the
successor of $\lambda=\beth_{\omega\cdot(\alpha+1)}$. The
$\delta_\alpha$ will be the cardinals at which we code
information, one bit each time, by forcing either the \GCH\
or its failure at $\delta_\alpha$. Every $\delta_\alpha$ is
a regular uncountable cardinal, and the sequence of
$\delta_\alpha$ is increasing, conveniently spaced apart in
order to avoid interference between the various levels of
coding. We could easily modify the argument to allow the
$\delta_\alpha$ to be spaced more closely together---and if
the \GCH\ holds in $V$, we could actually code at every
successor cardinal---but spacing the cardinals more
distantly as we have seems to produce the most transparent
general argument. Since $\lambda$ is not a limit of strong
limit cardinals, there is a largest strong limit cardinal
$\gamma$ below $\lambda$, namely
$\gamma=\beth_{\omega\cdot\alpha}$, and furthermore,
$\delta_\beta\leq\gamma^\plus$ for all $\beta<\alpha$. Note
also that $\alpha\le\beth_{\omega\cdot\alpha}=\gamma$, and
so $\alpha<\gamma^+<\lambda<\delta_\alpha$. Let $\Q_\alpha$
be the forcing that generically chooses whether to
force the \GCH\ or its negation at $\delta_\alpha$. %
To force the \GCH\ at $\delta_\alpha$, we use the forcing
$\Add(\delta_\alpha^\plus,1)$ to add a Cohen subset to
$\delta_\alpha^\plus$, which is equivalent to the canonical
forcing to collapse $2^{\delta_\alpha}$ to
$\delta_\alpha^\plus$. Forcing the failure of the \GCH\ at
$\delta_\alpha$, on the other hand, requires a little care
in the case that the \GCH\ fails below $\delta_\alpha$ in
$V$, for adding generic subsets to $\delta_\alpha$ may
collapse cardinals above $\delta_\alpha$, and even adding
$(\delta_\alpha^\plusplus)^V$ many subsets to
$\delta_\alpha$ may not in general suffice to ensure that
the \GCH\ fails at $\delta_\alpha$ in the extension. Since
the standard $\Delta$-system argument (see~\cite[Lemma
6.10]{Kunen:Independence}) establishes that
$\Add(\delta_\alpha,\theta)$ is
$(2^{\ltdelta_\alpha})^\plus$-c.c.~for any ordinal
$\theta$, however, it does in general suffice to add
$((2^{\ltdelta_\alpha})^\plusplus)^V$ many subsets to
$\delta_\alpha$. Thus, we take $\Q_\alpha$ as the
side-by-side forcing of these two alternatives, or in the
terminology of~\cite{Hamkins2000:LotteryPreparation}, the
poset $\Q_\alpha$ is the lottery sum $
\Add(\delta_\alpha^\plus,1)\oplus
\Add(\delta_\alpha,(2^{\ltdelta_\alpha})^\plusplus)$.\footnote{More
generally, for any family $\mathcal A$ of forcing notions, the
lottery sum $\oplus{\mathcal A}$ is defined to be
$\set{(\P,p)\st p\in\P\in{\mathcal A}}\union\singleton{\one}$,
ordered with $\one$ above everything and otherwise
$(\P,p)\leq(\Q,q)\leftrightarrow\P=\Q$ and $p\leq_{\P} q$. The
generic filter must in effect choose a single $\P\in{\mathcal
A}$ and force with it. For two posets, we use infix
notation: $\P\oplus\Q=\oplus\singleton{\P,\Q}$.} Conditions
opting for the first poset force the \GCH\ at
$\delta_\alpha$ and those opting for the second force its
failure. Note that $\Q_\alpha$ is $\ltdelta_\alpha$-closed
and has size $(2^{\ltdelta_\alpha})^\plusplus$, which is
strictly less than the next strong limit above
$\delta_\alpha$, which is itself less than
$\delta_{\alpha+1}$.

Let $\P$ be the class forcing product $\P=
\prod_\alpha\Q_\alpha$, with set support. That is,
conditions in $\P$ are set functions $p$, with
$\dom(p)\of\ORD$ and $p(\alpha)\in\Q_\alpha$, ordered by
extension of the domain and strengthening in each
coordinate. (In particular, we do not use Easton support,
which would not work here, because it would create new
unwanted definable subsets of the inaccessible cardinals,
if any.) Suppose that $G\of\P$ is $V$-generic and consider
the model $V[G]$. The class forcing $\P$ factors at every
ordinal $\alpha$ as $\P_\alpha\cross\P^\alpha$, where
$\P_\alpha=\prod_{\alpha'<\alpha}\Q_{\alpha'}$ and
$\P^\alpha=\prod_{\alpha'\geq\alpha}\Q_{\alpha'}$, where
again the products have set support, which for $\P_\alpha$
means full support. We claim that
$|\P_\alpha|<\delta_\alpha$ for every $\alpha$. To see
this, suppose that $\delta_\alpha=\lambda^\plus$, where
$\lambda$ is a strong limit cardinal, but not a limit of
strong limit cardinals. We have argued that
$\delta_\beta\leq\gamma^\plus$ for every $\beta<\alpha$,
where $\gamma$ is the largest strong limit cardinal below
$\lambda$. It follows that
$|\Q_\beta|\leq(2^{\ltgamma^\plus})^\plusplus$, and so
$\P_\alpha$ is the product of $\alpha$ many posets of at
most this size. Since $\alpha\leq\gamma^\plus$, this
implies that $\P_\alpha$ has size at most
$((2^{\ltgamma^\plus})^\plusplus)^{\gamma^\plus}$, and
since $\lambda$ is a strong limit cardinal, this is less
than $\lambda$ and hence less than $\delta_\alpha$, as
desired. Combining this with the fact that the tail forcing
$\P^\alpha$ is $\ltdelta_\alpha$-closed, it follows that
every set in $V[G]$ is added by some large enough initial
factor $\P_\alpha$. And using this, the standard arguments
show that $V[G]$ satisfies \ZFC. (For example, in the
terminology of~\cite{Reitz2006:Dissertation}, this is a
progressively closed product, and these always preserve
\ZFC.)

Let us verify that indeed the various levels of \GCH\
coding in our forcing do not interfere with each other. For
any ordinal $\alpha$, factor $\P$ as
$\P_\alpha\times\Q_\alpha\times\P^{\alpha+1}$. The tail
forcing $\P^{\alpha+1}$ is $\ltdelta_{\alpha+1}$-closed and
therefore does not affect the \GCH\ at $\delta_\alpha$. The
initial factor $\P_\alpha$ has size less than
$\delta_\alpha$, and therefore does not affect the \GCH\ at
$\delta_\alpha$. So the question whether the \GCH\ holds at
$\delta_\alpha$ in $V[G]$ is determined by what $G$ does on
$\Q_\alpha$. In other words, the overall \GCH\ pattern in
$V[G]$ on the cardinals $\delta_\alpha$ is determined in
accordance with the choices that $G$ makes in the
individual lotteries at each coordinate. A similar argument
shows that the forcing $\P$ preserves all strong limit
cardinals and creates no new strong limit cardinals. Thus,
the class $\set{\delta_\alpha\st\alpha\in\ORD}$ remains
definable in $V[G]$.

Let us now make the key observations about $V[G]$. First,
pushing upward, we claim that every set of ordinals in $V$
is coded into the \GCH\ pattern of $V[G]$. Suppose that $x$
is a set of ordinals in $V$ and $p$ is any condition in
$\P$. Choose ordinals $\beta$ and $\xi$ with $x\of\beta$
and $\dom(p)\of\xi$. Since $x$ is a set and $\P$ uses set
support, we may extend $p$ to a stronger condition $q\leq
p$ with $\dom(q)=\dom(p)\union [\xi,\xi+\beta)$, where $q$
opts on the interval $[\xi,\xi+\beta)$ to force the \GCH\
or its negation according to the pattern determined by $x$.
That is, we build $q$ so that for every $\alpha<\beta$, if
$\alpha\in x$, then $q(\xi+\alpha)$ opts to force the \GCH\
at $\delta_{\xi+\alpha}$, and if $\alpha\notin x$, then it
opts for its failure. Since we have argued that there is no
interference between the levels of coding, the condition
$q$ forces that the \GCH\ pattern in $V[G]$ for those
values of $\delta_{\xi+\alpha}$ is exactly the same as the
pattern of $x$ on $\beta$. Thus, it is dense that $x$ is
coded in this way, and so generically every set in $V$ will
be coded into the \GCH\ pattern of $V[G]$, since every set
in $V$ is coded by a set of ordinals in $V$. Since we have
mentioned that the class
$\set{\delta_\alpha\st\alpha\in\ORD}$ is definable in
$V[G]$, we may conclude immediately that every set in $V$
is ordinal definable in $V[G]$. For the generic \HOD, we
consider the set-forcing extensions of $V[G]$. Suppose that
$V[G][h]$ is obtained by further forcing $h\of\Q\in V[G]$.
Since the continuum function of $V[G][h]$ and $V[G]$ agree
above $|\Q|$, it follows that $V[G]$ and $V[G][h]$ have the
same strong limit cardinals and the same \GCH\ patterns
above $|\Q|$. This implies that a tail segment of
$\set{\delta_\alpha\st\alpha\in\ORD}$ remains definable in
$V[G][h]$, and the \GCH\ pattern on this segment is the
same in $V[G][h]$ as in $V[G]$. Since every set of ordinals
$x$ in $V$ was coded into the the \GCH\ pattern of $V[G]$
on the cardinals $\set{\delta_\alpha\st\alpha\in\ORD}$, a
simple padding argument shows that this implies that every
set in $V$ is in fact coded unboundedly often into the
\GCH\ pattern of $V[G]$ on these cardinals. So we conclude
that every set in $V$ remains ordinal definable in the
extension $V[G][h]$. Since the forcing $h$ was arbitrary,
we conclude $V\of\gHOD^{V[G]}$, and consequently also
$V\of\gHOD\of\gMantle\of\Mantle$.

Conversely, we now argue that the mantle $\Mantle $ of
$V[G]$ is contained in $V$. For any ordinal $\alpha$,
factor the forcing at $\alpha$ as
$\P=\P_\alpha\cross\P^\alpha$. The generic filter $G$
similarly factors as $G_\alpha\cross G^\alpha$. Since
$\P_\alpha$ is set forcing in $V$, it follows that the tail
extension $V[G^\alpha]$ is a ground of $V[G]$, and so the
mantle of $V[G]$ is contained within every $V[G^\alpha]$.
Since $\P^\alpha$ is ${<}\delta_\alpha$-closed, it follows
in particular that $V_{\alpha}^{V[G^\alpha]}=V_{\alpha}$,
and so $\bigcap_\alpha V[G^\alpha]=V$. Altogether, we have
established $V\of\gHOD\of\gMantle\of \Mantle\of V$, and so
all these are equal, as we claimed. Finally, let us
consider $\HOD^{V[G]}$. We have argued that
$V\of\HOD^{V[G]}$, and it remains for us to prove the
converse inclusion. For this, in order to control
$\HOD^{V[G]}$, it would be expected to appeal to
homogeneity properties of the forcing $\P$. Unfortunately,
the forcing $\P$ is not weakly homogeneous, because
different conditions can make fundamentally different
choices in the lotteries about how the forcing will
proceed. Nevertheless, we claim that there is sufficient
latent homogeneity in the forcing for an argument to
succeed. In order to show $\HOD^{V[G]}\of V$, it suffices
to show that $\HOD^{V[G]}$ is contained in every tail
extension $V[G^\alpha]$, as the intersection over these is
$V$, as we argued. Consider $V[G]$ as a forcing extension
of the tail extension $V[G^\alpha]$ by the initial forcing
$G_\alpha\of\P_\alpha$. Although $\P_\alpha$ is not weakly
homogenous, it is {\df densely weakly homogeneous}, meaning
that there is a dense set of conditions $q$ such that the
lower cone $\P_\alpha\restrict q$ is weakly homogeneous.
The point is simply that because we have used full support,
rather than Easton support, we may extend any condition in
$\P_\alpha$ to a condition with support $\alpha$.
Furthermore, we may extend to a condition that makes a
definite selection in each of the lotteries before stage
$\alpha$ as to which of the two posets should be used in
that coordinate. Since each of these individual posets is
weakly homogenous, the lower cone $\P_\alpha\restrict q$ is
the full product of weakly homogeneous forcing, and
consequently is itself weakly homogeneous. Thus, by
genericity, there is some $q\in G_\alpha$ such that
$\P_\alpha\restrict q$ is weakly homogenous. It follows
that every ordinal definable set of ordinals added by this
forcing is definable in the ground model $V[G^\alpha]$ from
ordinal parameters and the poset $\P_\alpha\restrict q$,
used as an additional parameter. So we have proved that
$\HOD^{V[G]}$ is included in every tail extension
$V[G^\alpha]$, and thus in $V$. So we have proved all the
desired equalities
$V=\Mantle^{V[G]}=\gMantle^{V[G]}=\gHOD^{V[G]}=\HOD^{V[G]}$.\end{proof}

\begin{theorem} Every model $V$ of \ZFC\ has a class forcing
extension $V[G]$ in which $V$ is the mantle, the generic
mantle and the generic \HOD, but $V[G]$
is the \HOD.\label{Theorem.V=M=gM=gHOD,V[G]=HOD}%
$$V=\Mantle^{V[G]}=\gMantle^{V[G]}=\gHOD^{V[G]},\qquad\text{but}\qquad \HOD^{V[G]}=V[G].$$
\end{theorem}

\begin{proof}  For this theorem, we must balance the various
forces on the classes differently, to keep the mantles and
the generic \HOD\ low, while allowing $\HOD^{V[G]}$ to
expand. Pushing the classes up at least to $V$, our
strategy will be once again to force that every set of
ordinals in $V$ is coded unboundedly into the continuum
function of $V[G]$. In order to push $\HOD^{V[G]}$ fully up
to $V[G]$, however, we will force that every new set of
ordinals in $V[G]$ is coded into the continuum function,
but these new sets will be coded each time only boundedly
often. This makes these sets ordinal definable in $V[G]$,
while allowing the factor argument of theorem
\ref{Theorem.V=M=gM=gHOD=HOD} to hold down the mantle to
$V$ and consequently also the generic mantle and generic
\HOD. The subtle effect is that the new sets become ordinal
definable, but only temporarily so, for further forcing can
erase the bounded coding and make them drop out of \HOD.

The essential component, for any regular cardinal $\kappa$,
is the {\df self-encoding forcing} at $\kappa$, which we
now describe. This is the forcing iteration $\Q$ of length
$\omega$ that begins by adding a Cohen subset of $\kappa$,
and then proceeds in each subsequent stage to code the
generic filter from the prior stage into the \GCH\ pattern
at the next block of cardinals. All coding will take place
in the interval $I=[\kappa,\lambda)$, where
$\lambda=\beth_\lambda$ is the least beth fixed point above
$\kappa$ (and the forcing will preserve all beth fixed
points). The end result in the corresponding extension
$V[G]$ is that the initial Cohen subset of $\kappa$ and the
entire generic filter $G$ is coded into the \GCH\ pattern
at cardinals in $I$. To be more precise, the forcing begins
at stage $0$ with $\Q_0=\Add(\kappa,1)$, adding a Cohen
subset $g_0\of\kappa_0=\kappa$. The stage $1$ forcing
$\dot\Q_1$ will code $g_0$ into the \GCH\ pattern at the
next $\kappa$ many cardinals. For this, we first force if
necessary to ensure that the \GCH\ holds at the next
$\kappa$ many cardinals (this may collapse cardinals, but
there is no need to exceed or even reach the beth fixed
point $\lambda=\beth_\lambda$), and then perform suitable
Easton forcing at these cardinals, so that for the next
$\kappa$ many cardinals $\nu$ above $\kappa$ in the
corresponding extension $V[g_0*g_1]$, we have either
$2^\nu=\nu^\plus$ or $2^\nu=\nu^\plusplus$, according to
whether the corresponding ordinal is in $g_0$. In order to
continue the iteration, we use a canonical pairing function
on ordinals in order to view $g_1$ as a subset of
$\kappa_1$, the supremum of the next $\kappa$ many
surviving cardinals above $\kappa$. In general, the generic
filter for the stage $n$ forcing $\Q_n$ is determined by a
subset $g_n\of\kappa_n$, and the stage $n+1$ forcing
$\dot\Q_{n+1}$ first forces if necessary the \GCH\ to hold
at the next $\kappa_n$ many cardinals, and then uses Easton
forcing to code $g_n$ into the \GCH\ pattern on those
cardinals. There is sufficient room to carry out each stage
of forcing below the next beth fixed point
$\lambda=\beth_\lambda$, and it is not difficult to see
that $\lambda=\sup_n \kappa_n$. The entire iteration $\Q$
consequently has size $\lambda^\omega$. The end result is
that if $G\of\Q$ is $V$-generic, then $G$ is coded
explicitly into the \GCH\ pattern of $V[G]$ on the interval
$I$. Note that $\Q$ is $\ltkappa$-closed and does not
affect the continuum function on cardinals outside the
interval $I$ provided the \GCH\ holds on $I$ in $V$.  In
the event that the \GCH\ fails in $V$ the situation is not
much worse, as $\Q$ does not affect the continuum function
outside the interval $(\gamma,\lambda^\omega)$ where
$2^\gamma \leq \kappa$.

We now assemble these components into the overall class
forcing. For every ordinal $\alpha$, let $\kappa_\alpha$ be
the $\alpha^{\rm th}$ cardinal of the form $(2^\lambda)^\plus$,
where $\lambda=\beth_\lambda$ is a beth fixed point. Let
$\Q_\alpha$ be the self-encoding forcing at
$\kappa_\alpha$, which adds a generic filter that encodes
itself into the \GCH\ pattern on the interval
$I_\alpha=[\kappa_\alpha,\lambda_\alpha)$, where
$\lambda_\alpha$ is the next beth fixed point above
$\kappa_\alpha$. Note that these intervals are disjoint.
Let $\P=\prod_\alpha\Q_\alpha$ be the set-support class
product of these posets. As in theorem
\ref{Theorem.V=M=gM=gHOD=HOD}, we may for any ordinal
$\alpha$ factor this forcing as $\P_\alpha\times\P^\alpha$,
where $\P_\alpha=\prod_{\beta<\alpha}\Q_\beta$ and
$\P^\alpha=\prod_{\beta\geq\alpha}\Q_\beta$, using again
set support in these products, which for $\P_\alpha$ means
full support. The initial factor $\P_\alpha$ is the product
of posets $\Q_\beta$ of size $\lambda_\beta^\omega$, for
$\beta<\alpha$, which therefore has size at most
$\lambda^\alpha\leq 2^\lambda$, where
$\lambda=\sup_{\beta<\alpha}\lambda_\beta$. This is
strictly less than $\kappa_\alpha$. In summary, we have
established the convenient factor properties that
$|\P_\alpha|<\kappa_\alpha$ and $\P^\alpha$ is
$\ltkappa_\alpha$-closed. The usual arguments now show that
if $G\of\P$ is $V$-generic, then $V[G]$ satisfies \ZFC\ and
every set in $V[G]$ is added by some stage $V[G_\alpha]$.
Also, for any ordinal $\alpha$, we may factor the forcing
as $\P_\alpha\times\Q_\alpha\times\P^{\alpha+1}$. Because
the final factor $\P^{\alpha+1}$ is
$\ltkappa_{\alpha+1}$-closed and the initial factor
$\P_\alpha$ has size less than $\kappa_\alpha$, neither of
these affects the \GCH\ pattern on the interval
$I_\alpha=[\kappa_\alpha,\lambda_\alpha)$. Thus, the \GCH\
pattern on $I_\alpha$ in $V[G]$ is determined by what the
generic filter $G$ does on the $\alpha^{\rm th}$ coordinate
$\Q_\alpha$.

We now observe that $V[G]$ exhibits the desired coding
features. If $x$ is any set of ordinals in $V$, then in any
coordinate stage $\alpha$ of forcing above $\sup(x)$, it is
dense for $x$ to appear as an interval in the generic
object added at the very first stage of forcing in
$\Q_\alpha$. The subsequent stages of forcing in
$\Q_\alpha$ will therefore have the effect of coding $x$
into the \GCH\ pattern in $I_\alpha$. Thus, the set $x$ is
coded into the \GCH\ pattern of $V[G]$, and is consequently
ordinal definable there. Since $x$ is coded unboundedly
often in this way, $x$ will remain ordinal definable in any
set-forcing extension of $V[G]$, because any such extension
$V[G][h]$ has the same \GCH\ pattern as $V[G]$ above the
size of the forcing $h$. Thus, $x$ is ordinal definable in
any such $V[G][h]$, and so $V\of\gHOD^{V[G]}$. We know in
general that $\gHOD\of\gMantle\of \Mantle$. In order to
complete the cycle, we now argue $\Mantle \of V$. Observe
that the tail forcing extension $V[G^\alpha]$ is a ground
of $V[G]$, because the initial factor $\P_\alpha$ is set
forcing. Thus, the mantle of $V[G]$ is contained within the
intersection of all $V[G^\alpha]$. But as before, the
intersection of all of the tail extensions $V[G^\alpha]$ is
simply $V$, by the increasing closedness of $\P^\alpha$,
and so the mantle of $V[G]$ is contained in $V$. This
establishes that $V\of\gHOD^{V[G]}\of\gMantle^{V[G]}\of
\Mantle^{V[G]}\of V$, and hence all are equal, as desired.

It remains to compute $\HOD^{V[G]}$. We have observed that
every set in $V[G]$ is added by some initial factor forcing
$\P_\alpha$ for some ordinal $\alpha$, and therefore every
object in $V[G]$ has the form $\tau_{G_\alpha}$, for some
ordinal $\alpha$ and some $\P_\alpha$-name $\tau$ in $V$.
Since we have already established that $V\of\HOD^{V[G]}$,
it follows that the name $\tau$ is ordinal definable in
$V[G]$. Furthermore, for every $\beta<\alpha$, the generic
filter $G(\beta)\of\Q_\beta$ added at stage $\beta$ is
coded into the continuum function on the interval
$I_\beta$, and in $V[G]$ we may definably assemble these
filters into the generic filter $G_\alpha$ on the product
$\P_\alpha=\prod_{\beta<\alpha}\Q_\beta$. Thus, $G_\alpha$
is also ordinal definable in $V[G]$. So $\tau_{G_\alpha}$
is ordinal definable in $V[G]$ and so $\HOD^{V[G]}=V[G]$,
as desired.\end{proof}

In the previous theorem, we have kept the mantles low,
while pushing up the \HOD. Next, in contrast, we keep the
{\HOD}s low, while pushing up the mantle.

\begin{theorem} The constructible universe $L$ has a
class-forcing extension $L[G]$ in which $L$ is the \HOD\
and generic \HOD, but $L[G]$ is the mantle.
\label{Theorem.L=gHOD=HOD,L[G]=M}%
$$L=\HOD^{L[G]}=\gHOD^{L[G]},\qquad\text{but}\qquad\Mantle^{L[G]}=L[G].$$
\end{theorem}

\begin{proof}  This theorem is a consequence of the main result of~\cite{HamkinsReitzWoodin2008:TheGroundAxiomAndVequalsHOD}. In
that article, Hamkins, Reitz and Woodin showed that if one
performs the Easton-support Silver iteration over $L$,
successively adding a Cohen subset to each regular
cardinal, then the resulting model $L[G]$ satisfies the
ground axiom. In other words, $L[G]$ has no nontrivial
grounds, and therefore is its own mantle. Since the forcing
is weakly homogeneous and ordinal-definable, it creates no
new ordinal-definable sets, and so $\HOD^{L[G]}=L$. Since
$\gHOD\of\HOD$, it follows that $\gHOD^{L[G]}=L$ as
well.\end{proof}

Unfortunately, we have not yet managed to determine the
generic mantle of $L[G]$.

\begin{question}
What is the generic mantle of the Hamkins-Reitz-Woodin
model $L[G]$?
\end{question}

Let us now prove a generalized version of theorem~\ref{Theorem.L=gHOD=HOD,L[G]=M}.

\begin{theorem} Every model $V$ of \ZFC\ has a class forcing
extension $V[G]$ in which $V$ is the $\HOD$ and generic
$\HOD$, but $V[G]$
is the mantle.\label{Theorem.V=HOD=gHOD,M=V[G]}%
$$V=\HOD^{V[G]}=\gHOD^{V[G]},\qquad\text{but}\qquad M^{V[G]}=V[G].$$
\end{theorem}

\begin{proof}  For this theorem, we adapt the main argument and
result of
\cite{HamkinsReitzWoodin2008:TheGroundAxiomAndVequalsHOD}. The
forcing takes place in two steps. First, we perform the
forcing of theorem~\ref{Theorem.V=M=gM=gHOD=HOD}, coding
into the \GCH\ pattern on the cardinals
$\delta_\alpha=\beth_{\omega\cdot\alpha+\omega}^\plus$.
This forcing, the reader will recall, is the set-support
product $\prod_\alpha\Q_\alpha$, where $\Q_\alpha$
generically chooses whether to force the \GCH\ at
$\delta_\alpha$ or its failure. If $V[H]$ is the resulting
forcing extension, then the argument of theorem
\ref{Theorem.V=M=gM=gHOD=HOD} shows that
$V=\HOD^{V[H]}=\gHOD^{V[H]}=\Mantle^{V[H]}=\gMantle^{V[H]}$.

The second step is to perform an Easton-support Silver
iteration $\P$ over $V[H]$, a class length forcing
iteration that at stage $\alpha$ adds a Cohen subset to
$(2^{\delta_\alpha})^\plus$. These cardinals lay
conveniently within the interval
$(\delta_\alpha,\delta_{\alpha+1})$ between the successive
coding points of the first step. Suppose that $K\of\P$ is
$V[H]$-generic, and we consider the final extension
$V[H][K]$. Because the Silver iteration does not affect the
\GCH\ coding at the various $\delta_\alpha$ and preserves
all strong limit cardinals and therefore the definability
of the class $\set{\delta_\alpha\st\alpha\in\ORD}$, it
follows that every set in $V$ is coded arbitrarily highly
in the continuum function of $V[H*K]$. This coding survives
into set generic extensions of $V[H*K]$ and their $\HOD$s,
and so $V\of \gHOD^{V[H*K]}$. Conversely, the Silver
iteration $\P$ is weakly homogeneous and ordinal definable,
and so $\HOD^{V[H*K]}\of\HOD^{V[H]}=V$. We conclude
$V=\HOD^{V[H*K]}=\gHOD^{V[H*K]}$.

We claim now $V[H*K]$ has no nontrivial grounds, and is
therefore its own mantle. In other words, $V[H*K]$
satisfies the ground axiom. This part of the argument
follows the main method and result of
\cite{HamkinsReitzWoodin2008:TheGroundAxiomAndVequalsHOD}. Let
us suppose towards a contradiction that $W$ is a nontrivial
ground, so that $V[H*K]=W[h]$ for some nontrivial forcing
$h\of\Q\in W$ over $W$. We may assume that conditions in
$\Q$ are ordinals less than $|\Q|^W$. Since this is set
forcing, the model $W$ has the same continuum function as
$W[h]$ on cardinals above $|\Q|^{W}$. In particular, these
models eventually have the same strong limit cardinals, and
so a final segment of the class
$\set{\delta_\alpha\st\alpha\in\ORD}$ is definable in $W$.
Because the sets of ordinals in $V$ are all coded
unboundedly often into the \GCH\ pattern of $V[H*K]=W[h]$
on these cardinals, it follows that they are also coded in
$W$ and so $V\of W$. Let $\lambda$ be the least strong
limit above $|\Q|^W$, so that $\lambda^\plus=\delta_\alpha$
for some ordinal $\alpha$. Factor both forcing notions at
$\alpha$, resulting in $V[H*K]=V[H_1][H_2][K_1][K_2]$,
where $H_1$ is the part of $H$ at coordinates
$\beta<\alpha$ and $H_2$ is the part of $H$ at coordinates
$\beta\geq\alpha$, and $K_1\of\P_\alpha$ is the part of $K$
below stage $\alpha$ and $K_2\of\P_{\alpha,\infty}$ is the
tail part of the $K$ iteration from stage $\alpha$ onwards.
Thus, we have $V[H_1][H_2][K_1][K_2]=W[h]$. Note that the
forcing adding $H_1$ and $K_1$ has size strictly less than
$\lambda$. Let $\delta$ be a regular cardinal above the
size of these initial factors and also $|\Q|^W$, but less
than $\lambda$. Since the forcing that adds $H_2$ is
$\ltdelta$-closed in $V$, every $\delta$-small subset of
$H_2$ in $V[H][K]$ is covered by a condition in $V$. To see
this, let $x\of H_2$ have size less than $\delta$ in
$V[H][K]$. By the closedness of the last factor, $x\in
V[H][K_1]=V[H_2][H_1][K_1]$. Since $H_1*K_1$ is generic for
forcing of size less than $\delta$ in $V[H_2]$, we can
conclude by the $\delta$-cover property (see lemma
\ref{Lemma.ClosurePointForcing}) that there is $x'\in
V[H_2]$ such that $x\of x'$ and $x'$ has size less than
$\delta$ in $V[H_2]$. Since $H_2$ is a class in $V[H_2]$,
we may pick $x'$ in such a way that $x'\of H_2$. Now it
follows by the ${<}\delta$-closedness of $\P_2$ in $V$ that
$x'\in V$. And by genericity of $H_2$, there is a common
strengthening of all the conditions in $x'$ in $V$. Let's
call it a master condition for $x$. Since $V\of W$, such
conditions are also in $W$. Thus, every $\delta$-small
approximation $H_2\intersect B$ to $H_2$, with $B\in W$ of
size less than $\delta$ in $W$, will be the set of
conditions in $B$ which are weaker than a master condition
for $H_2\intersect B$, and will therefore itself be in $W$.
Since the forcing adding $h$ over $W$ is small relative to
$\delta$ and therefore exhibits the $\delta$-approximation
property by lemma~\ref{Lemma.ClosurePointForcing}, it
follows that $H_2$ is amenable to $W$, and hence that
$V[H_2]\of W$. In $W$, let $A\of\lambda$ code
$(2^\ltlambda)^W$, as well as $\Q$-names $\dot H_1,\dot
K_1\in W$ such that $H_1=(\dot H_1)_h$ and $K_1=(\dot
K_1)_h$; for example, we could simply ensure that
$H_\lambda^W\in L[A]$. Since $A\in W$, it follows that
$V[H_2][A]\of W$. Since $A$ codes all bounded subsets of
$\lambda$ in $W$, we have $\Q\in V[H_2][A]$. The filter $h$
is $V[H_2][A]$-generic for $\Q$, and we may consider the
forcing extension $V[H_2][A][h]$. By the choice of $A$,
this model has the names $\dot H_1$ and $\dot K_1$,
allowing it to build $H_1$ and $K_1$ using $h$, and so
$V[H][K_1]\of V[H_2][A][h]$. Conversely, the objects $A$
and $h$ could not have been added by the tail forcing $K_2$
over $V[H][K_1]$, since this forcing is $\leqlambda$-closed
over $V[H][K_1]$, and so $V[H_2][A][h]\of V[H][K_1]$.
Consequently, $V[H_2][A][h]=V[H][K_1]$. We may therefore
view the model $V[H*K]$ as having arisen by forcing over
$V[H_2][A]$, adding $h*K_2$, since
$V[H_2][A][h*K_2]=V[H_2][A][h][K_2]=V[H][K_1][K_2]=V[H*K]$.
Alternatively, we may view $V[H*K]=W[h]$ as having arisen
by forcing over $W$, adding $h$. Each of these extensions
exhibits the $\delta$-approximation and cover properties.
Furthermore, since $P(\delta)^W=P(\delta)^{V[H_2][A]}$ and
the cardinal successor of $\delta$ is the same in each of
the models at hand, it follows by lemma
\ref{Lemma.ApproximationCoverImpliesW=W'} that $W=V[H_2][A]$.
Thus, $W[h]=V[H_2][A][h]$, which we already established was
$V[H][K_1]$, contrary to our assumption that
$W[h]=V[H][K]$. So there can be no such ground
$W$.\end{proof}

\begin{theorem} Every model $V$ of \ZFC\ has a class forcing
extension $V[G]$ in which $V[G]$ is its own mantle, generic
mantle, generic \HOD\ and
\HOD.\label{Theorem.V[G]=M=gM=gHOD=HOD}%
$$V[G]=\Mantle^{V[G]}=\gMantle^{V[G]}=\gHOD^{V[G]}=\HOD^{V[G]}$$
\end{theorem}

\begin{proof}  The main result of~\cite{Reitz2006:Dissertation},
using ideas of
\cite{McAloon1971:ConsistencyResultsAboutOrdinalDefinability},
shows that every model $V$ of \ZFC\ has a class forcing
extension $V[G]$ in which every set of ordinals is coded
unboundedly often into the continuum function. (This
assertion was called the continuum coding axiom.) In this
case, it is easy to see that
$V[G]=\Mantle^{V[G]}=\gMantle^{V[G]}=\gHOD^{V[G]}=\HOD^{V[G]}$,
since these sets remains coded in this way in all
set-forcing extensions and grounds of such extensions by
set forcing.\end{proof}

\begin{remark}\rm
We have chosen to use \GCH\ coding in the arguments above.
However, the arguments are easily modified to accommodate
other methods of coding, which would be consistent with
\GCH\ in the target model. For example, coding via
$\Diamond_\kappa^*$, in the style of
\cite{Brooke-Taylor:Thesis,Brooke-Taylor2009:LargeCardinalsAndDefinableWellOrders},
seems perfectly acceptable. The result would be that one
could add $V[G]\satisfies\GCH$ to each of theorems
\ref{Theorem.V=M=gM=gHOD=HOD},
\ref{Theorem.V=M=gM=gHOD,V[G]=HOD},
\ref{Theorem.V=HOD=gHOD,M=V[G]}, and \ref{Theorem.V[G]=M=gM=gHOD=HOD}.
\end{remark}

\section{Inner mantles and the outer core}

The mantle $\Mantle$ of the universe $V$ arises by brushing
away the outermost layers of forcing like so much
accumulated dust and sand to reveal the underlying ancient
structure, the mantle, on which these layers rest. If the
mantle is itself a model of \ZFC, then it makes sense to
penetrate still deeper, computing the mantle of the mantle
and so on, revealing still more ancient layers, and one
naturally desires to iterate the process. We begin easily
enough with $\Mantle^0=V$ and then recursively define the
successive inner mantles by
$\Mantle^{n+1}=\Mantle^{\Mantle^n}$, that is,
$\Mantle^{n+1}$ is the mantle of $\Mantle^n$, provided that
this is a model of \ZFC. Thus, $\Mantle^1$ is the mantle
and $\Mantle^2$ is the mantle of the mantle, and so on. One
would naturally expect to continue this recursion
transfinitely with an intersection
$\Mantle^\omega=\Intersect_{n<\omega}\Mantle^n$ at
$\omega$, but here an interesting and subtle
metamathematical obstacle rises up, preventing a simple
success. The issue is that although we have provided
definitions of the mantle $\Mantle^1$ and of the
mantle-of-the-mantle $\Mantle^2$ and so on, these
definitions become increasingly complex as the procedure is
iterated,  and an observant reader will notice that our
recursive definition does not actually provide a definition
of the $n^{\rm th}$ mantle $\Mantle^n$ that is uniform in $n$.
Instead, it is a recursion that takes place in the
meta-theory rather than the object theory, and so on this
definition we may legitimately refer to the $n^{\rm th}$ mantle
$\Mantle^n$ only for meta-theoretic natural numbers $n$. In
particular, it does not provide a uniform definition of the
inner mantles $\Mantle^n$, and consequently with it we seem
unable to perform the intersection $\Intersect_n\Mantle^n$
in a definable manner.

Exactly the same issue arises when one attempts to iterate
the class $\HOD$, by considering the $\HOD$-of-$\HOD$ and
so on in the iterated $\HOD^n$ models. A 1974 result of
Harrington appearing in
~\cite[section 7]{Zadrozny1983:IteratingOrdinalDefinability},
with related work in~\cite{McAloon1974:OnTheSequenceHODn},
shows that it is relatively consistent with \Godel-Bernays
set theory that $\HOD^n$ exists for each $n<\omega$ but the
intersection $\HOD^\omega=\Intersect_n\HOD^n$ is not a
class. There simply is in general no uniform definition of
the classes $\HOD^n$. We expect an analogous result for the
iterated mantles $\Mantle^n$, and this is a current focus
of study for us.

Nevertheless, some models of set theory have a special
structure that allows them to enjoy a uniform definition of
these classes, and in these models we may continue the
iteration transfinitely. Following the treatment of the
iterated $\HOD^\alpha$ and $\gHOD^\alpha$ of
\cite{HamkinsKirmayerPerlmutter2012:GeneralizationsOfKunenInconsistency},
and working in \Godel-Bernays set theory, we define that a
\GB\ class $\bar M$ is a {\df uniform presentation} of the
inner mantles $\Mantle^\alpha$ for $\alpha<\eta$ if $\bar
M\of \set{(x,\alpha)\st \alpha<\eta}$ and the slices $\bar
M^\alpha=\set{x\st (x,\alpha)\in \bar M}$ for $\alpha<\eta$
are all inner models of \ZFC\ and obey the defining
properties of the iterated mantle construction, namely, the
base case $\bar M^0=V$, the successor case $\bar
M^{\alpha+1}=\Mantle^{\bar M^\alpha}$ and the limit case
$\bar M^\gamma=\Intersect_{\alpha<\gamma}\bar M^\alpha$ for
limit ordinals $\gamma$. By induction, any two such classes
$\bar M$ agree on their common coordinates, and when such a
class $\bar M$ has been provided we may legitimately and
unambiguously refer to the $\alpha^{\rm th}$ mantle
$\Mantle^\alpha$. We accordingly define the phrase ``{\df
the $\eta^{\rm th}$ inner mantle $\Mantle^\eta$ exists}'' to
mean that $\eta$ is an ordinal and there is a uniform
presentation $\bar M$ of the inner mantles $\Mantle^\alpha$
for $\alpha\leq\eta$. It is easy to see that the $n^{\rm th}$
inner mantle $\Mantle^n$ exists for any (meta-theoretic)
natural number $n$, and if $\Mantle^\eta$ exists, so does
$\Mantle^\alpha$ for any $\alpha<\eta$. But as with the
Harrington result in the case of $\HOD^\alpha$, we do not
expect necessarily to be able to proceed through limit
ordinals or even up to $\omega$ uniformly. Note that even
when the $\eta^{\rm th}$ inner mantle $\Mantle^\eta$ exists,
there seems little reason to expect that it is necessarily
a definable class, even when $\eta$ is definable or
comparatively small, such as $\eta=\omega$.

\begin{question}
Is there is a model of \GBC\ in which there is no uniform
presentation of the inner mantles $\Mantle^n$ for
$n\lt\omega$? In particular, is there a model where the
$\omega^{\rm th}$ inner mantle $\Mantle^\omega$ does not exist?
Is there a model where $\Mantle^\omega$ exists as a class
but does not satisfy \ZFC?
\end{question}

In analogy with the situation with the iterated $\HOD$s, we
expect affirmative answers to these questions.

In the case that there is a uniform presentation of the
inner mantles $\Mantle^\alpha$ for all ordinals $\alpha$,
that is, presented uniformly by a single \GB\ class $\bar
M\of V\times\ORD$, then it follows by corollary
\ref{Corollary.IntersectionGivesZF} that
$\Mantle^{\ORD}=\Intersect_\alpha\Mantle^\alpha\satisfies\ZFC$
as well. That is, if one can iteratively and uniformly
compute the inner mantles all the way through $\ORD$, then
the intersection $\Mantle^{\ORD}$ of these inner mantles is a \ZFC\ model,
and we may continue past $\ORD$ by defining
$\Mantle^{\ORD+1}=\Mantle^{\Mantle^{\ORD}}$, and so on.

One way that this might happen, of course, is if the inner
mantle process actually stabilizes before $\ORD$, that is,
if for some $\alpha$ we have $\Mantle^\alpha=\Mantle^\beta$
for all $\beta>\alpha$; in this case, $\Mantle^\alpha$ is a
model of \ZFC, but not a nontrivial forcing extension of
any deeper model (or equivalently, a model of \ZFC\ plus
the ground axiom). If $\alpha$ is minimal with this
property, then we say that the sequence of inner mantles
{\df stabilizes} at $\alpha$ and refer to the model
$\Mantle^\alpha$ as the {\df outer core} of the universe in
which it was computed. This is what remains when all outer
layers of forcing have been successively stripped away.
There seems to be no reason in general why the inner mantle
process should necessarily stabilize at an early stage, or
even after iterating through all the ordinals. For example,
there seems to be no obvious reason why $M_{\ORD}$ cannot itself
be a forcing extension of some other still-deeper model.
Indeed, we make the following strong negative conjecture
about this prospect.

\begin{conjecture}
Every model of \ZFC\ is the $\Mantle^{\ORD}$ of another
model of \ZFC\ in which the sequence of inner mantles does
not stabilize. More generally, every model of \ZFC\ is the
$\Mantle^\alpha$ of another model of \ZFC\ for any desired
$\alpha\leq\ORD$, in which the sequence of inner mantles
does not stabilize before
$\alpha$.\label{Conjecture.InnerMantles}
\end{conjecture}

Let us explain the sense in which this conjecture, along
with the main theorem of this article, can be viewed as
philosophically negative. If one has adopted the
philosophical view that beneath the set-theoretic universe
lies a highly regular structure, some kind of canonical
inner model, which may have become obscured over the eons
by the accumulated layers of subsequent forcing
constructions over that structure, then one would be led to
expect that the mantle, or perhaps the inner mantles or the
outer core, which sweep away all these subsequent layers of
forcing, would exhibit highly regular structural features.
But since our main theorem shows that {\it every} model of
\ZFC\ is the mantle of another model, one cannot prove in
general that the mantle exhibits any extra structural
features at all. And under conjecture
\ref{Conjecture.InnerMantles}, the same can be said of the
inner mantles and the outer core. In particular, if the
conjecture holds, then there are models whose outer core is
realized first at stage $\alpha$ in $\Mantle^\alpha$, for
any desired $\alpha\leq\ORD$, since every model of
$\ZFC+\GA$ would be the $(\Mantle^\alpha)^{\Vbar}$ of a
suitable extension $\Vbar$. Thus, under the conjecture one
should not expect to prove universal regularity features
for the outer core even when it is realized as
$\Mantle^{\ORD}$, beyond what one can prove about arbitrary
models of the ground axiom. And this is not much, since
\cite{HamkinsReitzWoodin2008:TheGroundAxiomAndVequalsHOD} shows
that models of \GA\ need not even satisfy $V=\HOD$.

Note that if the conjecture holds, then the inner mantle
$\Mantle^{\ORD}$ can be itself a forcing extension (since
many models of \ZFC\ are forcing extensions), and the
process of stripping away the outer layers of forcing has
not terminated even with $\ORD$ many iterations. Thus, the
conjecture implies that there can be models of set theory
having no outer core. It is also conceivable that the inner
mantle calculation might break down at some ordinal stage
$\alpha$ for the reason that $\Mantle^\alpha$ no longer
satisfies \ZFC, and in this case, the original model would
have no outer core.

\begin{question}
Under what circumstances does the outer core exist?
\end{question}

We may similarly carry out the entire construction using
generic mantles in place of mantles. Thus, we define that
``the $\eta^{\rm th}$ inner generic mantle $\gMantle^\eta$
exists,'' for $\eta\leq\ORD$, if there is a class $\bar
M\of V\times\eta$, whose slices are inner models of \ZFC\
respecting the iterative definition of the generic mantle,
so that $\bar M^0=V$ at the beginning, $\bar
M^{\alpha+1}=\gMantle^{\bar M^\alpha}$ as successor
ordinals, and $\bar
M^\lambda=\Intersect_{\alpha<\lambda}\bar M^\alpha$ at
limits. If the process terminates by reaching a model of
$\ZFC$ that is equal to its own generic mantle---thereby
satisfying $V=\gMantle^V$, the {\df generic ground
axiom}---then we refer to this model as the {\df generic
outer core}. One may similarly consider the iterated inner
$\HOD$ construction and inner generic $\HOD$s, as in
\cite{HamkinsKirmayerPerlmutter2012:GeneralizationsOfKunenInconsistency}.
We extend conjecture~\ref{Conjecture.InnerMantles}
naturally to the following.

\begin{conjecture}
If\/ $V$ is any model of \ZFC, then for any
$\alpha\leq\ORD$, there is another model $\Vbar$ of \ZFC\
in which $V$ is the $\alpha^{\rm th}$ inner mantle, the
$\alpha^{\rm th}$ generic inner mantle, the $\alpha^{\rm th}$ inner
\HOD\ and the $\alpha^{\rm th}$ inner generic \HOD.
$$V=(\Mantle^\alpha)^{\Vbar}=(\gMantle^\alpha)^\Vbar=(\HOD^\alpha)^\Vbar=(\gHOD^\alpha)^\Vbar$$
\end{conjecture}

\section{Large cardinal indestructibility across the
multiverse}

Laver~\cite{Laver78} famously proved that any supercompact
cardinal $\kappa$ can be made indestructible by further
$\ltkappa$-directed closed forcing. In this section, we
consider the possibility of large cardinal
indestructibility not just in the upwards direction, but
throughout the relevant portion of the generic multiverse.

For any class $\Gamma$ of forcing notions, we say that $W$
is a {\df $\Gamma$-ground} of $V$, or equivalently, that
$V$ is a {\df $\Gamma$-extension} of $W$, if $V=W[G]$ for
some $W$-generic $G\of\P\in W$, where $\P$ has property
$\Gamma$ in $W$. The {\df $\Gamma$-generic multiverse} is
obtained by closing the universe under $\Gamma$-extensions
and $\Gamma$-grounds.

\begin{theorem} If $\kappa$ is supercompact, then there is a
class forcing extension $V[G]$ in which $\kappa$ remains
supercompact, becomes indestructible by further
$\ltkappa$-directed closed forcing, and the ground axiom
holds. Thus, $\kappa$ remains supercompact in all
$\ltkappa$-directed closed extensions and (vacuously) in
all $\ltkappa$-directed closed
grounds.\label{Theorem.SupercompactIndestructibleGA}
\end{theorem}

\begin{proof}  Suppose that $\kappa$ is supercompact. By Laver's
theorem~\cite{Laver78}, we may find a set-forcing extension
$V[g]$ in which $\kappa$ remains supercompact and becomes
indestructible by all $\ltkappa$-directed closed forcing.
By coding the universe into the continuum function above
$\kappa$, we may perform $\ltkappa$-directed closed class
forcing to find a model $V[g][G]$ in which every set is
coded unboundedly often into the \GCH\ pattern of
$V[g][G]$. It follows that $V[g][G]$ satisfies the
continuum coding axiom and therefore also the ground axiom,
as in
\cite{Reitz2007:TheGroundAxiom,Reitz2006:Dissertation}.
Because the \GCH\ coding forcing is $\ltkappa$-directed
closed (and factors easily into set forcing followed by
highly closed forcing), it follows that the
supercompactness of $\kappa$ remains indestructible over
$V[g][G]$ by further $\ltkappa$-directed closed
forcing.\end{proof}

By combining the previous argument with the usual methods
to attain global indestructibility, one may also accomplish
this effect for many supercompact cardinals simultaneously.
We show next that it is not possible to improve this to the
entire $\ltkappa$-directed closed generic multiverse.

\begin{theorem}\label{Theorem.SupercompactNeverIndestructibleInMultiverse}
No supercompact cardinal $\kappa$ is
indestructible throughout the $\ltkappa$-directed closed
generic multiverse. Indeed, for every cardinal $\kappa$,
there is a $\ltkappa$-directed closed generic ground in
which $\kappa$ is not measurable.
\end{theorem}

\begin{proof} Fix any cardinal $\kappa$. Let $\P$ be the forcing
to add a slim $\kappa$-Kurepa tree, a tree $T$ of height
$\kappa$, whose $\alpha^{\rm th}$ level (for infinite $\alpha$)
has size at most $\Card{\alpha}$, but which has $\kappa^\plus$
many $\kappa$-branches. The natural way to do this has the
feature that $\P$ is $\ltkappa$-closed, but not
$\ltkappa$-directed closed. An easy reflection argument
shows that no measurable cardinal can have a slim
$\kappa$-Kurepa tree, and so $\kappa$ is not measurable in
$V[T]$. But the forcing $\P$ can be absorbed into the
forcing $\Coll(\kappa,2^\kappa)$, since the combined
forcing $\P*\dot\Coll(\kappa,2^\kappa)$ is
$\ltkappa$-closed, has size $2^\kappa$ and collapses
$2^\kappa$ to $\kappa$, and $\Coll(\kappa,2^\kappa)$ is the
unique forcing notion (up to isomorphism of the complete
Boolean algebras) with this property. Thus, if
$G\of\Coll(\kappa,2^\kappa)$ is $V[T]$-generic for this
collapse, we may rearrange the combined generic filter
$T*G$ as a single $V$-generic filter
$H\of\Coll(\kappa,2^\kappa)$ such that $V[T][G]=V[H]$.
Since the collapse forcing is $\ltkappa$-directed closed,
it follows that $V[T]$ is a $\ltkappa$-directed closed
generic ground of $V$, in which $\kappa$ is not measurable,
as desired.
\end{proof}

The situation for indestructible weak compactness is
somewhat more attractive.

\begin{theorem}
If $\kappa$ is supercompact, then there is a set-forcing
extension $V[G]$ in which $\kappa$ is weakly compact and
this is indestructible throughout the $\ltkappa$-closed
generic multiverse of $V[G]$.
\end{theorem}

\begin{proof} By the Laver preparation~\cite{Laver78}, there is a
forcing extension $V[G]$ in which the supercompactness of
$\kappa$ is indestructible by all $\ltkappa$-directed
closed forcing. By~\cite[Theorem
22]{HamkinsJohnstone2010:IndestructibleStrongUnfoldability},
this implies that the weak compactness of $\kappa$ is
indestructible over $V[G]$ by all $\ltkappa$-closed
forcing, that is, generalizing $\ltkappa$-directed closed
to $\ltkappa$-closed forcing.

\begin{sublemma}
If $\kappa$ is weakly compact and this is indestructible by
$\ltkappa$-closed forcing, then $\kappa$ retains this
property throughout the $\ltkappa$-closed generic
multiverse.
\end{sublemma}

\begin{proof}
Suppose that $\kappa$ is weakly compact in a model $\Vbar$
and indestructible over $\Vbar$ by all $\ltkappa$-closed
forcing. Clearly, this property remains true in all
$\ltkappa$-closed extensions, since the two-step iteration
of $\ltkappa$-closed forcing is $\ltkappa$-closed. We also
argue conversely, however, that the property remains true
in all $\ltkappa$-closed grounds. Suppose that $W$ is a
$\ltkappa$-closed ground of $\Vbar$, so that $\Vbar=W[g]$
is $W$-generic for some $\ltkappa$-closed forcing
$g\of\P\in W$. First, we know that $\kappa$ is inaccessible
in $W$ because inaccessibility is downward absolute.
Second, in order to establish the tree property for
$\kappa$ in $W$, suppose that $T$ is a $\kappa$-tree in
$W$. By the tree property of $\kappa$ in $\Vbar$, there is
a branch $b$ through $T$ in $\Vbar=W[g]$. Thus, $W$ has a
$\P$-name $\dot b$ for this branch. In $W$, since the
forcing $\P$ is $\ltkappa$-closed, we may build a
pseudo-generic descending $\kappa$-sequence of conditions
in $\P$ that decide more and more of this name $\dot b$.
Thus, in $W$ we can build a branch through $T$, thereby
establishing this instance of the tree property for
$\kappa$ in $W$. It follows that $\kappa$ is weakly compact
in $W$. Consider now any $\ltkappa$-closed forcing $\Q\in
W$. Since $\P$ is $\ltkappa$-closed in $W$, it follows that
$\Q$ remains $\ltkappa$-closed in $\Vbar=W[g]$. Thus, for
any $W[g]$-generic filter $h\of\Q$, it follows by the
indestructibility of $\kappa$ over $\Vbar$ that $\kappa$ is
weakly compact in $\Vbar[h]=W[g][h]$. Since $\P$ also
remains $\ltkappa$-closed in $W[h]$, it follows by the
pseudo-forcing downward absoluteness argument that $\kappa$
is weakly compact in $W[h]$. Thus, there can be no
condition in $\Q$ forcing over $W$ that $\kappa$ is not
weakly compact in the extension of $W$ by $\Q$. Thus, we
have proved that the weak compactness of $\kappa$ is
indestructible over $W$ by $\ltkappa$-closed forcing. Thus,
the indestructibility property we desire is preserved as
one moves through the $\ltkappa$-closed generic
multiverse.\end{proof}

The theorem follows directly.\end{proof}

In the case of small forcing, the well-known \Levy-Solovay
theorem of~\cite{LevySolovay67} establishes that every
measurable cardinal $\kappa$ is preserved by the move to
any $\kappa$-small forcing extension as well as to any
$\kappa$-small ground (that is, in each case, by forcing of
size less than $\kappa$). It follows that any measurable
cardinal $\kappa$ is measurable throughout the
corresponding small-generic multiverse. Similarly, most of
the other classical large cardinals are preserved to small
forcing extensions and grounds, and therefore retain their
large cardinal property throughout the small-generic
multiverse. It is natural to consider also the extent
to which other set-theoretic axioms, such as the proper
forcing axiom, are indestructible throughout
significant portions of the generic multiverse. For example, \PFA\ is known to be indestructible to ${\lt}\aleph_2$-directed closed forcing extensions. Is it also necessarily preserved to all ${\lt}\aleph_2$-directed closed ground models? Of course, this situation is at least relatively consistent, since we may make it a vacuous claim, as in theorem~\ref{Theorem.SupercompactIndestructibleGA}, by forcing the \CCA\ over any model of \PFA\ via ${\lt}\aleph_2$-directed closed forcing, resulting in a model of $\ZFC+\PFA+\GA$, which has no nontrivial grounds at all. Meanwhile, \PFA\ is not generally preserved to proper grounds, since the usual forcing from a model with a supercompact cardinal is proper and could start in a model without \PFA. We wonder whether there is a robust fragment of the generic multiverse in which \PFA\ or other forcing axioms might be necessarily invariant.

\bibliographystyle{alpha}
\bibliography{MathBiblio,HamkinsBiblio}
\end{document}

%% file: MathMacrosJDH.tex
\newtheorem{theorem}{Theorem}
\newtheorem{maintheorem}[theorem]{Main Theorem}
\newtheorem{corollary}[theorem]{Corollary}
\newtheorem{sublemma}{Lemma}[theorem]
\newtheorem{lemma}[theorem]{Lemma}

\newtheorem{question}[theorem]{Question}

\newtheorem{observation}[theorem]{Observation}

\newtheorem{conjecture}[theorem]{Conjecture}
\newtheorem{fact}[theorem]{Fact}
\newtheorem{definition}[theorem]{Definition}

\newtheorem{remark}[theorem]{Remark}

\newcommand{\QED}{\end{proof}}

\def\proclaim[#1]{{\bf #1}}
\def\BF#1.{{\bf #1.}}

\newcommand{\Vopenka}{Vop\v{e}nka}

\newcommand{\Godel}{G\"odel}

\newcommand{\Levy}{L\'{e}vy}

\newcommand{\B}{{\mathbb B}}

\newcommand{\D}{{\mathbb C}}

\renewcommand{\P}{{\mathbb P}}
\newcommand{\Q}{{\mathbb Q}}

\newcommand{\R}{{\mathbb R}}

\newcommand{\Vbar}{{\overline{V}}}

\newcommand{\one}{\mathbbm{1}} 
\newcommand{\of}{\subseteq}

\newcommand{\fo}{\supseteq}

\newcommand{\set}[1]{\{\,{#1}\,\}}
\newcommand{\singleton}[1]{\left\{{#1}\right\}}

\newcommand{\dom}{\mathop{\rm dom}}

\newcommand{\cof}{\mathop{\rm cof}}

\newcommand{\Add}{\mathop{\rm Add}}

\newcommand{\Coll}{\mathop{\rm Coll}}

\newcommand{\plus}{{+}}
\newcommand{\plusplus}{{{+}{+}}}

\newcommand{\restrict}{\upharpoonright} 
\newcommand{\satisfies}{\models}
\newcommand{\forces}{\Vdash}

\newcommand{\possible}{\mathop{\raisebox{-1pt}{$\Diamond$}}}

\newcommand{\Mantle}{{\mathord{\rm M}}}
\newcommand{\gMantle}{{\mathord{\rm gM}}}  
\newcommand{\gHOD}{\ensuremath{\mathord{{\rm g}\HOD}}} 
\newcommand{\cross}{\times}

\newcommand{\union}{\cup}

\newcommand{\Union}{\bigcup}

\newcommand{\intersect}{\cap}
\newcommand{\Intersect}{\bigcap}

\newcommand{\smalllt}{\mathrel{\mathchoice{\raise2pt\hbox{$\scriptstyle<$}}{\raise1pt\hbox{$\scriptstyle<$}}{\raise0pt\hbox{$\scriptscriptstyle<$}}{\scriptscriptstyle<}}}
\newcommand{\smallleq}{\mathrel{\mathchoice{\raise2pt\hbox{$\scriptstyle\leq$}}{\raise1pt\hbox{$\scriptstyle\leq$}}{\raise1pt\hbox{$\scriptscriptstyle\leq$}}{\scriptscriptstyle\leq}}}
\newcommand{\lt}{\smalllt}

\newcommand{\ltkappa}{{{\smalllt}\kappa}}

\newcommand{\leqlambda}{{{\smallleq}\lambda}}
\newcommand{\ltlambda}{{{\smalllt}\lambda}}
\newcommand{\ltgamma}{{{\smalllt}\gamma}}

\newcommand{\leqbeta}{{{\smallleq}\beta}}
\newcommand{\leqdelta}{{{\smallleq}\delta}}
\newcommand{\ltdelta}{{{\smalllt}\delta}}

\newcommand{\Card}[1]{{\left|#1\right|}}
\newcommand{\boolval}[1]{\mathopen{\lbrack\!\lbrack}\,#1\,\mathclose{\rbrack\!\rbrack}}
\def\[#1]{\boolval{#1}}
\newbox\gnBoxA
\newdimen\gnCornerHgt
\setbox\gnBoxA=\hbox{\tiny$\ulcorner$}
\global\gnCornerHgt=\ht\gnBoxA
\newdimen\gnArgHgt
\def\gcode #1{%
\setbox\gnBoxA=\hbox{$#1$}%
\gnArgHgt=\ht\gnBoxA%
\ifnum     \gnArgHgt<\gnCornerHgt \gnArgHgt=0pt%
\else \advance \gnArgHgt by -\gnCornerHgt%
\fi \raise\gnArgHgt\hbox{\tiny$\ulcorner$} \box\gnBoxA %
\raise\gnArgHgt\hbox{\tiny$\urcorner$}}
\newcommand{\UnderTilde}[1]{{\setbox1=\hbox{$#1$}\baselineskip=0pt\vtop{\hbox{$#1$}\hbox to\wd1{\hfil$\sim$\hfil}}}{}}
\newcommand{\Undertilde}[1]{{\setbox1=\hbox{$#1$}\baselineskip=0pt\vtop{\hbox{$#1$}\hbox to\wd1{\hfil$\scriptstyle\sim$\hfil}}}{}}
\newcommand{\undertilde}[1]{{\setbox1=\hbox{$#1$}\baselineskip=0pt\vtop{\hbox{$#1$}\hbox to\wd1{\hfil$\scriptscriptstyle\sim$\hfil}}}{}}
\newcommand{\UnderdTilde}[1]{{\setbox1=\hbox{$#1$}\baselineskip=0pt\vtop{\hbox{$#1$}\hbox to\wd1{\hfil$\approx$\hfil}}}{}}
\newcommand{\Underdtilde}[1]{{\setbox1=\hbox{$#1$}\baselineskip=0pt\vtop{\hbox{$#1$}\hbox to\wd1{\hfil\scriptsize$\approx$\hfil}}}{}}

\newcommand{\st}{\mid}

\renewcommand{\implies}{\mathrel{\rightarrow}}

\renewcommand{\iff}{\mathrel{\leftrightarrow}}

\def\<#1>{\langle#1\rangle}

\newcommand{\RO}{\mathop{{\rm RO}}}
\newcommand{\ORD}{\mathop{{\rm Ord}}}
\newcommand{\Ord}{\mathop{{\rm Ord}}}

\newcommand{\ZFC}{{\rm ZFC}}
\newcommand{\ZF}{{\rm ZF}}

\newcommand{\KM}{{\rm KM}}
\newcommand{\GB}{{\rm GB}}
\newcommand{\GBC}{{\rm GBC}}

\newcommand{\GCH}{{\rm GCH}}

\newcommand{\DC}{{\rm DC}}

\newcommand{\GA}{{\rm GA}}
\newcommand{\DDG}{{\rm DDG}}

\newcommand{\PFA}{{\rm PFA}}

\newcommand{\HOD}{{\rm HOD}}

\newcommand{\IMH}{{\rm IMH}}

\newcommand{\card}{{{\rm card}}}

\newcommand{\cell}[1]{\boxit{\hbox to 17pt{\strut\hfil$#1$\hfil}}}
\newcommand{\head}[2]{\lower2pt\vbox{\hbox{\strut\footnotesize\it\hskip3pt#2}\boxit{\cell#1}}}
\newcommand{\boxit}[1]{\setbox4=\hbox{\kern2pt#1\kern2pt}\hbox{\vrule\vbox{\hrule\kern2pt\box4\kern2pt\hrule}\vrule}}
\newcommand{\Col}[3]{\hbox{\vbox{\baselineskip=0pt\parskip=0pt\cell#1\cell#2\cell#3}}}
\newcommand{\tapenames}{\raise 5pt\vbox to .7in{\hbox to .8in{\it\hfill input: \strut}\vfill\hbox to
.8in{\it\hfill scratch: \strut}\vfill\hbox to .8in{\it\hfill output: \strut}}}
\newcommand{\Head}[4]{\lower2pt\vbox{\hbox to25pt{\strut\footnotesize\it\hfill#4\hfill}\boxit{\Col#1#2#3}}}
\newcommand{\Dots}{\raise 5pt\vbox to .7in{\hbox{\ $\cdots$\strut}\vfill\hbox{\ $\cdots$\strut}\vfill\hbox{\
$\cdots$\strut}}}
\newcommand{\df}{\it} 